%% file: bilevelArxiv.tex
\documentclass[11pt]{article}
\pdfoutput=1

\usepackage[utf8]{inputenc}
\usepackage[T1]{fontenc}

\usepackage[verbose=true,
letterpaper,
textheight=8.4in,
textwidth=6.1in,
margin=1in,
headheight=12pt,
headsep=25pt,
footskip=30pt]{geometry}

\usepackage[toc,page,header]{appendix}
\usepackage{titletoc}

\usepackage{microtype}
\usepackage{mathtools}
\usepackage{booktabs}
\usepackage[usenames,dvipsnames]{xcolor}   
\usepackage[round]{natbib}
\usepackage{graphicx}
\usepackage{amsmath,amssymb,amsfonts,amsxtra,bm}
\usepackage{amsthm}
\usepackage{xspace}
\usepackage{paralist}
\usepackage{enumerate}
\usepackage{enumitem}
\usepackage{hyperref}
\usepackage{url}
\usepackage{colortbl}
\usepackage{makecell,multirow}       
\usepackage{nicefrac}       
\usepackage{thmtools}  
\usepackage{xspace}
\usepackage{footnote, tablefootnote}
\usepackage{epstopdf}
\usepackage{latexsym}
\usepackage{graphicx}
\usepackage{bbm}
\usepackage{comment}
\usepackage{algorithm}
\usepackage{algpseudocode}
\usepackage[most]{tcolorbox}

\titlecontents{section}[3em]{\vspace{-1pt}}%
    {\bfseries\contentslabel{2em}}
    {}
    {\titlerule*[0.5pc]{.}\contentspage}%
    
\titlecontents{subsection}[5em]{\vspace{-1pt}}%
    {\contentslabel{2em}}
    {}
    {\titlerule*[0.5pc]{.}\contentspage}%

\titlecontents{subsubsection}[5em]{\vspace{-1pt}}%
    {\footnotesize\contentslabel{3em}}
    {}
    {\titlerule*[0.5pc]{.}\contentspage}%

\numberwithin{equation}{section}
\usepackage{ss}
\usepackage{tikz}
\input{arXivMacros}

\begin{document}

\title{Riemannian Bilevel Optimization}

\author{\name{Sanchayan Dutta}\email{dutta@ucdavis.edu}\\
\addr{UC Davis, Davis, CA, USA}\\
\name{Xiang Cheng} \email{chengx@mit.edu}\\
\addr{Massachusetts Institute of Technology, Cambridge, USA}\\[2pt]
\name{Suvrit Sra} \email{s.sra@tum.de}\\
\addr{TU Munich, Garching, Germany}
}

\maketitle

\begin{abstract}
We develop new algorithms for Riemannian bilevel optimization. We focus in particular on batch and stochastic gradient-based methods, with the explicit goal of avoiding second-order information such as Riemannian hyper-gradients. We propose and analyze \(\algo\), a method that leverages first-order gradient information to navigate the complex geometry of Riemannian manifolds efficiently. Notably, $\algo$ is a single-loop algorithm, and thus easier to implement and use. Under various setups, including stochastic optimization, we provide explicit convergence rates for reaching \(\epsilon\)-stationary points. We also address the challenge of optimizing over Riemannian manifolds with constraints by adjusting the multiplier in the Lagrangian, ensuring convergence to the desired solution without requiring access to second-order derivatives.
\end{abstract}

\section{Introduction}
We investigate Riemannian bilevel optimization problems described by:
\begin{equation}
\label{eq:P}
\min_{x \in \mathcal M}\  F(x) := f\left(x, y^{*}(x)\right) \quad \text{s.t.} \quad y^{*}(x) \in \argmin_{y \in \mathcal N} g(x, y), \tag{\textbf{P}}
\end{equation}
where \(\Mc\) and \(\Nc\) are $d_x$- and $d_y$-dimensional complete Riemannian manifolds, respectively, and \(f\) and \(g\) are smooth functions. The function \(F\) serves as the outer objective, with \(g\) as the inner objective and \(y^*(x)\) as the optimal solution for the inner problem.

Bilevel optimization provides a useful model for hierarchical decision-making, and is thus of great value to various fields such as machine learning, economics, operations research, and engineering. In machine learning it is directly relevant to applications such as meta-learning \citep{rajeswaran2019meta,hospedales2021meta,pham2020contextual,ravi2016optimization}, hyper-parameter optimization \citep{franceschi2018bilevel,bao2021stability,pedregosa2016hyperparameter}, model selection \citep{kunapuli2008bilevel,giovannelli2021bilevel}, architecture search \citep{liu2018darts,wang2022zarts,zhang2021idarts}, and reinforcement learning \citep{konda1999actor,sutton2018reinforcement,hong2023two}. Algorithms like the two-timescale stochastic approximation (TTSA) \citep{hong2023two} highlight the ongoing development of efficient solutions for these problems.

Riemannian optimization arises in several applications, e.g., policy optimization, where algorithms utilize the Fisher Information manifold \citep{ding2020natural,kakade2001natural,cen2022fast}, and matrix factorization, where problems are reformulated over suitable matrix manifolds \citep{ahn2021riemannian,hou2020fast,li2021weakly}. Hyperbolic manifolds have also been applied in neural network architecture design \citep{peng2021hyperbolic} and image segmentation \citep{ghadimi2018approximation}. 

Toward solving~\eqref{eq:P} we build on recent progress in (Euclidean) bilevel optimization~\citep{kwon2023fully}, and develop Riemannian ``fully first-order'' batch and stochastic methods. Indeed, while some aspects of the Euclidean analysis translate directly into the Riemannian setting, curvature causes distortion that poses unique challenges in developing the analysis. 


\subsection{Related Work}

\subsubsection*{Riemannian optimization}

The development of efficient stochastic gradient-based optimization algorithms is crucial across various domains, with first-order methods preferred for their computational efficiency. Several first-order stochastic methods have been adapted to the Riemannian setting, such as Riemannian stochastic gradient descent (RSGD) \citep{bonnabel2013stochastic}, and AMSGRAD \citep{becigneul2018riemannian}. Convergence analysis in geodesically convex cases has been proposed \citep{zhang2016first}, and nonconvex Riemannian optimization methods have been developed to handle complex landscapes \citep{kasai2018riemannian, hu2024riemannian}. Saddle-point problems, including geodesic min-max formulations, have also been explored \citep{zhang2023sion}.

\medskip

\subsubsection*{Bilevel optimization}

The motivation for bilevel optimization on manifolds arises from addressing lower-level problems that are non-strongly convex, frequently encountered outside traditional Euclidean spaces. \cite{chen2023bilevel} highlights the necessity for geometry-aware optimization techniques, particularly useful in meta-learning and hyperparameter tuning scenarios. Additional research focuses on constrained bilevel problems in machine learning, where model parameters are defined on specific manifolds \citep{lin2008riemannian, franceschi2018bilevel, tabealhojeh2023rmaml}. Many of these problems involve non-convex constraints, which existing methods often cannot adequately address. Moreover, even those methods that can handle such constraints may not fully exploit the computational advantages provided by a geometry-aware approach \citep{beck2023computationally}.

\subsubsection*{Riemannian bilevel optimization}

\cite{li2024riemannian} investigated hypergradient calculations for bilevel optimization on Riemannian manifolds, proposing deterministic and stochastic algorithms (RieBO and RieSBO). Similarly, \cite{han2024framework} developed a framework for bilevel optimization, offering several hypergradient estimation strategies and conducting convergence and complexity analyses. Both works emphasize the importance of hypergradient information. In contrast, our approach is fully first-order, avoiding second-order information like Riemannian hypergradients, presenting a novel direction in this field.

Useful Riemannian bilevel problems often involve manifold constraints, such as machine learning with manifold constraints or robust PCA on the Stiefel manifold \citep{yao2024constrained, hong2023two, xu2023efficient, khanduri2023linearly, podosinnikova2014robust}. Our paper contributes tools for manifold-based bilevel optimization, opening avenues in optimal transport and geometric ODEs where manifold structures are crucial \citep{figalli2011necessary, udriste2020geometric}.

\subsection{Main Contributions}

The structure of our analysis parallels the fully-first order Euclidean approach \citep{kwon2023fully}. We reformulate the optimization problem (\ref{eq:P}) into a constrained, single-level problem:
\[
\min_{x \in \Mc, y \in \Nc} f(x, y) \quad \text{s.t.} \quad g(x, y) - g^*(x) \leq 0,
\tag{\textbf{P'}}\label{eqn:P'}\]
where \( g^*(x) = g(x, y^*(x)) \). The associated Lagrangian \( \mathcal{L}_{\lambda}(x, y) = f(x, y) + \lambda ( g(x, y) - g^*(x) ) \) uses a multiplier \( \lambda > 0 \). To optimize \( \mathcal{L}_{\lambda} \), we employ Riemannian gradient descent, calculating gradients with first-order derivatives, akin to the Euclidean approach.

The main challenge in addressing \(\eqref{eqn:P'}\) is choosing \(\lambda\). The optimal solution \(x^{*} = \argmin _{x} F(x)\) is found as \(\lambda \to \infty\). However, a high \(\lambda\) makes \(\mathcal{L}_{\lambda}(x, y)\) non-smooth, affecting gradient-descent efficacy.

To address this, we begin with \(\lambda = \lambda_{0} > 0\) and increment it gradually. Each iteration \(k\) sets \(\lambda_{k} = O\left(k^{b}\right)\) for \(b \in (0,1]\), balancing bias removal and nonsmoothness increase rate. This strategy is vital for converging to an \(\epsilon\)-stationary point of \(F\) without needing second derivatives.

One primary contribution is the derivation of an optimal growth rate for \( \lambda_k \), ensuring non-asymptotic convergence to an \(\epsilon\)-stationary point of \( F \) without needing second-order derivatives. Algorithm \(\algo\) advances Riemannian bilevel optimization by employing a first-order gradient method that navigates curvature complexities such as varying sectional curvatures, parallel transport, and the geometry of geodesics. These complexities impact the behavior of gradients and necessitate specialized techniques to ensure efficient convergence. Our method converges efficiently to an \(\epsilon\)-stationary solution, as outlined in Theorem~\ref{thm:main}.

\begin{theorem}[Informal]
\label{t:informal}
There exist choices of hyperparameters of Algorithm \algo such that the following stationarity guarantees hold:
\begin{enumerate}
    \item \label{item:noise_both}If noise is present in $\grad f$ and $\grad g$, then $\mathbb{E}\left[\| \grad F(x_{K}) \|^{2}\right] = \tilde{O}(K^{-2/7})$;
    \item \label{item:noise_f_only}If noise is present only in $\grad f$, then $\mathbb{E}\left[\| \grad F(x_{K}) \|^{2}\right] = \tilde{O}(K^{-2/5})$; and 
    \item \label{item:exact_gradients}If both $\grad f$ and $\grad g$ are exact, then $\mathbb{E}\left[\| \grad F(x_{K}) \|^{2}\right] = \tilde{O}(K^{-2/3})$.
\end{enumerate}
\end{theorem}

\noindent

Theorem \ref{t:informal} is an informal version of our convergence guarantee; the formal version is Theorem \ref{thm:main}. 

Below, we summarize the key aspects of our result:

\begin{itemize}[leftmargin=0.4cm]
    \item \textbf{Stochastic First-Order Algorithm Without Hypergradient Computations}: 
    Algorithm \algo uniquely employs only gradient computations, avoiding the complex hypergradient calculations seen in prior works like \citep{ghadimi2018approximation}. This simplification is especially beneficial in large-scale machine learning tasks, where data completeness is not guaranteed, reducing the need for extensive iterations and streamlining the optimization process in stochastic settings.

    \item \textbf{Convergence Rate Guarantees in Stochastic Scenarios}:
    The convergence rates for Algorithm \algo in stochastic gradient scenarios are $\tilde{O}(\epsilon^{-3.5})$ when both $\grad f$ and $\grad g$ are noisy, improving to $\tilde{O}(\epsilon^{-2.5})$ when only $\grad f$ is noisy, and $\tilde{O}(\epsilon^{-1.5})$ under exact gradients. These rates are tight and align with those expected in Euclidean optimization, adapted to the complexities of Riemannian manifolds.

    \item \textbf{Achieving $\epsilon$-Stationarity}:
    The algorithm effectively converges to an $\epsilon$-stationary point, where the norm of the gradient is below $\epsilon$. This capability is critical for assessing the efficacy of optimization algorithms under various gradient noise conditions.

    \item \textbf{Modular Analysis of $\lambda$}:
    Our analysis reveals how different adjustments in $\lambda$ affect step size, noise variance, and bias, providing insights that help optimize algorithm performance on Riemannian manifolds. This modular approach allows for the strategic modification of $\lambda$, enhancing both computational efficiency and algorithm robustness, paving the way for future advancements in manifold-based optimization algorithm design.
\end{itemize}

\section{Mathematical Background}

\subsection{Some Precepts of Riemannian Geometry}

The Hessian of a function \(f\) at a point \(p\) on a manifold \(\Mc\) with a metric \(g\) is defined using the Levi-Civita connection \(\nabla\). It is a bilinear form that can be expressed in local coordinates as:
\begin{equation}
\Hess_f(X,Y) = X(Y(f)) - (\nabla_X Y)(f),
\end{equation}
where \(X\) and \(Y\) are vector fields on \(\Mc\), and \(\nabla_X Y\) is the covariant derivative of \(Y\) in the direction of \(X\).

The term \(\nabla_{xy}^2\) can be related to the components of the \(\Hess\) in local coordinates. Specifically, if \(X\) and \(Y\) are coordinate vector fields corresponding to coordinates \(x\) and \(y\) respectively, then \(\nabla_{xy}^2 f\) would correspond to the \((x,y)\)-component of the Hessian matrix of \(f\), which is:

\begin{equation}
\Hess_f(X,Y) = \nabla_X \nabla_Y f - \nabla_{\nabla_X Y} f 
\end{equation}

In local coordinates, this would be written as:

\begin{equation}
\Hess_f(\partial_x, \partial_y) = \frac{\partial^2 f}{\partial x \partial y} - \Gamma_{xy}^k \frac{\partial f}{\partial k}
\end{equation}

where \(\Gamma_{xy}^k\) are the Christoffel symbols of the second kind, which encode the manifold's connection and hence its curvature.

\subsection{Main Definitions and Assumptions}
\label{ss:assumptions}

\begin{definition}[\(\epsilon\)-stationary point]
A point \(x \in \Mc\) is called \(\epsilon\)-stationary if \(\|\grad F(x)\|_x^{2} \leq \epsilon\). A stochastic algorithm is said to achieve an \(\epsilon\)-stationary point in \(K\) iterations if \(\mathbb{E}\left[\left\|\grad F\left(x_{K}\right)\right\|_{x_K}^{2}\right] \leq \epsilon\), where the expectation is over the algorithm's stochasticity.
\end{definition}

\textbf{Notation:} \(O_{\mathrm{P}}(\cdot)\) denotes the order of constants dependent on instance-specific parameters (e.g., Lipschitz constants, strong convexity, and smoothness conditions). The notation \(a_{k} \asymp b_{k}\) indicates that \(a_{k}\) and \(b_{k}\) decrease or increase at the same rate as \(k \rightarrow \infty\), i.e., \(\lim _{k \rightarrow \infty} a_{k} / b_{k} = \Theta(1)\). The norm \(\|\cdot\|_x\) is induced by the metric at \(x\), reflecting the geometry of \(\Mc\).

To outline the class of problems (\ref{eq:P}) of interest, we assume the outer-level objective's optimal value on \(\Mc\) is bounded below by \(F^{*} \coloneqq \arg \min _{x \in \Mc} F(x)>-\infty\).

\begin{assumption}[Objective Functions Properties]
\label{assumption:objective_functions}
The objective functions \(f\) and \(g\) exhibit the following properties:
\begin{enumerate}
    \item \(f\) is continuously differentiable on \(\Mc\). Its gradient satisfies \(l_{f, 1}\)- smoothness, meaning that for any two points \(x, y\) on \(\Mc\),
    \begin{equation}
    \| \text{PT}_{y \leftarrow x} \nabla f(x) - \nabla f(y) \|_{\Mc} \leq l_{f, 1} d_{\Mc}(x, y).
    \end{equation}
    \item \(g\) is continuously differentiable on \(\mathcal N\). Its gradient satisfies \(l_{g, 1}\)-smoothness, implying,
    \begin{equation}
    \| \text{PT}_{y \leftarrow x} \nabla g(x) - \nabla g(y) \|_{\Mc} \leq l_{g, 1} d_{\Mc}(x, y).
    \end{equation}
    \item For every \(\bar{x} \in \Mc\), the magnitude of the gradient \(||\nabla_{y} f(\bar{x}, y)||_{\Mc}\) is bounded by \(l_{f, 0}\) for all \(y\).
\end{enumerate}
\end{assumption}

\begin{assumption}[Lower-level Objective Properties]
\label{assumption:lower_level}
For the lower-level objective \(g\) on \(\Nc\):
\begin{enumerate}
    \item For every \(\bar{x} \in \Mc\), the function \(g(\bar{x}, y)\) is \(\mu_{g}\)-strongly convex in \(y\) for some \(\mu_{g}>0\) on \(\Nc\).
    \item \(g\) is twice continuously differentiable on \(\Nc\), and its Hessian \(\nabla^{2} g\) satisfies \(l_{g, 2}\)-Lipschitz continuity,
    \begin{equation}
    \| \text{PT}_{y \leftarrow x}^{\gamma} \nabla^2 g(y) \text{PT}_{x \leftarrow y}^{\gamma} - \nabla^2 g(x) \|_{\Nc} \leq l_{g, 2} d_{\Nc}(x,y).
    \end{equation}
\end{enumerate}
\end{assumption}

\begin{assumption}[Gradient Access]
\label{assumption:gradient_access}
Access to the gradients of the objective functions \(f\) and \(g\) is provided via unbiased estimators \(\grad f(x, y ; \zeta)\) and \(\grad g(x, y ; \phi)\), where:
\begin{equation}
\begin{aligned}
\mathbb{E}[\grad f(x, y ; \zeta)] &= \grad f(x, y), \\
\mathbb{E}[\grad g(x, y ; \phi)] &= \grad g(x, y),
\end{aligned}
\end{equation}
and the variances of the stochastic gradient estimators are bounded:
\begin{equation}
\begin{aligned}
\mathbb{E}\left[\left\|\grad f(x, y ; \zeta)-\grad f(x, y)\right\|_x^{2}\right] &\leq \sigma_{f}^{2}, \\
\mathbb{E}\left[\left\|\grad g(x, y ; \phi)-\grad g(x, y)\right\|_x^{2}\right] &\leq \sigma_{g}^{2},
\end{aligned}
\end{equation}
where \(\|\cdot\|_x\) denotes the norm induced by the metric at point \(x\).
\end{assumption}

\begin{assumption}[Gradient Boundedness]
\label{assumption:gradient_boundedness}
The gradients with respect to \(x\) for \(f\) and \(g\) are bounded for every \(\bar{y}\), with \(\|\grad_x f(x, \bar{y})\|\) and \(\|\grad_x g(x, \bar{y})\|\) bounded by \(l_{f, 0}\) and \(l_{g, 0}\), respectively, for all \(x\).
\end{assumption}

\begin{assumption}[Second-order Smoothness of \(f\)]
\label{assumption:f_second_order}
\(f\) is twice continuously differentiable, with its Hessian \(\text{Hess} f\) being \(l_{f, 2}\)-Lipschitz continuous in the sense over the product of the manifold's tangent spaces at \((x, y)\).
\end{assumption}

The assumptions \ref{assumption:objective_functions} through \ref{assumption:f_second_order} are necessary for guaranteeing the smoothness of \(y_{\lambda}^{*}(x)\) and the efficacy of the inner iterations throughout all outer iterations. These assumptions align the analysis with the inherent curvature and metrics of \(\mathcal{M}\) and \(\mathcal{N}\). They are essential for our proof of Theorem \ref{thm:main}.

\subsection{Computing the Hypergradient via Perturbation Analysis}

In this section, we utilize first-order perturbations in the variables \(x\) and \(y\) to derive the hyper-gradient \(\grad F(x)\) of \(F\) at \(x\). This formulation of the hyper-gradient is crucial for the proofs of the foundational lemmas and theorems that follow.

Considering an infinitesimal perturbation \(\delta v\) within the tangent space \(T_x\Mc\), we transition to a new manifold point \(x' = \Exp_x(\delta v)\). Similarly, perturbing the solution \(y^*(x)\) by \(\delta u\) in \(T_y\Nc\) leads to \(y' = \Exp_y(\delta u)\).

The first-order Taylor expansion of \(g\) around the point \((x, y^*(x))\) is expressed as:
\begin{equation}
\grad_y g(x', y^*(x)) \approx \grad_y g(x, y^*(x)) + \grad_x (\grad_y g)(x, y^*(x))[\delta v].
\end{equation}
Incorporating the perturbation \(\delta u\) in $T_y\Nc$, we refine our approximation to:
\begin{equation}
\grad_{y'} g(x', y') \approx \grad_y g(x', y^*(x)) + \Hess_{yy} g(x, y^*(x))[\delta u].
\end{equation}

To satisfy the optimality condition \(\grad_{y'} g(x', y') = 0\) for \(y'\) as the new minimizer, we establish a linkage between \(\delta u\) and \(\delta v\):
\begin{equation}
\grad_x (\grad_y g)(x', y^*(x))[\delta v] + \Hess_{yy} g(x, y^*(x))[\delta u] = 0.
\end{equation}
Solving for \(\delta u\), we invert the Hessian of \(g\) with respect to \(y\), obtaining:
\begin{equation}
[\delta u] = -(\Hess_{yy} g(x', y^*(x)))^{-1} \grad_x (\grad_y g)(x', y^*(x))[\delta v].
\end{equation}

Finally, the gradient of \(F\) at \(x\), influenced by the movements \(\delta v\) and \(\delta u\), is concisely articulated as:
\begin{equation}
\grad F(x) = \grad_x f(x, y^*(x)) - \Hess_{xy} g(x, y^*(x)) (\Hess_{yy} g(x, y^*(x)))^{-1} \grad_y f(x, y^*(x)),
\label{Fhypergradient}
\end{equation}
culminating our systematic approach to compute the hyper-gradient via perturbation analysis.

\section{Algorithm Design and Step-Size Calculations}

\subsection{Algorithm}

We devise an algorithm to find a stationary point of the bilevel problem, specifically, a point where $F(x)=f(x, y^{*}(x))$ is stationary, using gradients of $f$ and $g$. Considering the formulation $(\mathbf{P}^{\prime})$ and aiming to bypass second-order derivatives, we assess the gradient of $\mathcal{L}_{\lambda}$:
\begin{equation}
\label{eq:gradient_equations}
\begin{aligned}
& \grad_x \mathcal{L}_{\lambda}(x, y) = \grad_x f(x, y) + \lambda \left( \grad_x g(x, y) - \grad g^{*}(x) \right), \\
& \grad_y \mathcal{L}_{\lambda}(x, y) = \grad_y f(x, y) + \lambda \grad_y g(x, y).
\end{aligned}
\end{equation}
On the manifold, the gradient of $g^{*}(x)$ simplifies to $\grad g^{*}(x) = \grad_x g(x, y^{*}(x))$ due to $g$'s optimality at $y^{*}(x)$. To optimize $\mathcal{L}_{\lambda}(x, y)$, we introduce an auxiliary variable $z$, approximating $y^{*}(x)$, and consider an alternative bilevel formulation (\ref{eq:P}) with the outer-level objective $\mathcal{L}_{\lambda}(x', z)$, where $x'=(x, y)$ is the outer variable, and $z$ is the inner variable. This modification alters $F(x)$'s landscape, introducing a bias that must be managed carefully to not affect the function $\mathcal{L}_{\lambda}$'s smoothness, which is crucial for step-size and noise variance.

To manage the bias, we explore the relation between $\mathcal{L}_{\lambda}$ and $F(x)$ through an auxiliary function $\mathcal{L}_{\lambda}^{*}$ defined as:
\begin{equation}
\label{eq:minimization}
\mathcal{L}_{\lambda}^{*}(x) := \min_{y} \mathcal{L}_{\lambda}(x, y).
\end{equation} 

For $\lambda > 2 l_{f, 1} / \mu_{g}$, $\mathcal{L}_{\lambda}(x, y)$ becomes strongly convex in $y$, ensuring a unique minimizer $y_{\lambda}^{*}(x)$:
\begin{equation}
\label{eq:argminimization}
y_{\lambda}^{*}(x) := \argmin_{y} \mathcal{L}_{\lambda}(x, y).
\end{equation}

Given $F(x) = \lim_{\lambda \to \infty} \mathcal{L}_{\lambda}^{*}(x)$ for any $x$ in $X$, $\mathcal{L}_{\lambda}^{*}(x)$ effectively approximates $F(x)$ for a sufficiently large $\lambda$. This approach is underpinned by a lemma adapted for manifolds.

\begin{algorithm}
\caption{\algo - Riemannian First-order Fast Stochastic Approximation}\label{alg:RF2SA}
\begin{algorithmic}[1]
\State \textbf{Input:} step sizes: $\{\alpha_k, \gamma_k\}$, multiplier difference sequence: $\{\delta_k\}$, inner-loop iteration count: $T$,
\State step-size ratio: $\xi$, initializations: $\lambda_0, x_0, y_0, z_0$
\State \text{For} $k = 0$ to $K-1$ \text{do}
    \State \hspace{\algorithmicindent} $z_{k,0} \gets z_{k}$, $y_{k,0} \gets y_{k}$
    \State \hspace{\algorithmicindent} \text{For} $t = 0$ to $T-1$ \text{do}
        \State \hspace{\algorithmicindent}\hspace{\algorithmicindent} $z_{k,t+1} \gets \Exp_{z_{k,t}}(- \gamma_k h_{gz}^{k,t})$
        \State \hspace{\algorithmicindent}\hspace{\algorithmicindent} $y_{k,t+1} \gets \Exp_{y_{k,t}}(- \alpha_k (h_{fy}^{k,t} + \lambda_k h_{gy}^{k,t}))$
    \State \hspace{\algorithmicindent} \text{EndFor}
    \State \hspace{\algorithmicindent} $z_{k+1} \gets z_{k,T}$, $y_{k+1} \gets y_{k,T}$
    \State \hspace{\algorithmicindent} $x_{k+1} \gets \Exp_{x_k}(- \xi \alpha_k (h_{fx}^k + \lambda_k (h_{gxy}^k - h_{gxz}^k)))$
    \State \hspace{\algorithmicindent} $\lambda_{k+1} \gets \lambda_k + \delta_k$
\State \text{EndFor}
\end{algorithmic}
\end{algorithm}

\begin{lemma}
\label{lemma:riemannian_gradient_bound}
For any $x \in X$ and $\lambda \geq 2 l_{f, 1} / \mu_{g}$, the gradient of $\mathcal{L}_{\lambda}^{*}(x)$ is
\begin{equation}
\grad_{x} \mathcal{L}_{\lambda}(x, y_{\lambda}(x)) = \grad_{x} f(x, y_{\lambda}^{}(x)) + \lambda \left( \grad_{x} g(x, y_{\lambda}^{}(x)) - \grad_{x} g(x, y^{}(x)) \right).
\end{equation}
Furthermore, the norm of the difference between the gradients of $F(x)$ and $\mathcal{L}_{\lambda}^{*}(x)$ is bounded by
\begin{equation}
|\grad F(x) - \grad \mathcal{L}_{\lambda}^{*}(x)| \leq \frac{C_{\lambda}}{\lambda}.
\end{equation}
where $C_{\lambda} := \frac{4 l_{f, 0} l_{g, 1}}{\mu_{g}^2}\left(l_{f, 1} + \frac{2 l_{f, 0} l_{g, 2}}{\mu_{g}}\right)$.
\end{lemma}

The gradient $\grad \mathcal{L}_{\lambda}^{*}(x)$ is computable with first-order derivatives of both $f$ and $g$. Thus, any first-order method locating a stationary point of $\mathcal{L}_{\lambda}^{*}(x)$ approximates the trajectory of $x$ updated with $\grad F(x)$, with a bias of $O(1/\lambda)$.

We use $\grad \mathcal{L}_{\lambda}^{*}(x)$ as a proxy for $\grad F(x)$ to produce a sequence of iterates $\{x_{k}\}$. Concurrently, we generate sequences $\{y_{k}\}$ and $\{z_{k}\}$ to approximate the solutions $y_{\lambda_{k}}^{*}(x_{k})$ and $y^{*}(x_{k})$, respectively, incrementing $\lambda_{k}$ with $k$ to ensure the bias in $\{x_{k}\}$ diminishes to zero.

Our Fully First-order Stochastic Approximation (F$^2$SA) method, adapted for the manifold, employs stochastic gradients as unbiased estimators:
\begin{equation}
\begin{aligned}
h_{g z}^{k, t} & := \grad_{y} g(x_{k}, z_{k, t} ; \phi_{z}^{k, t}), \quad h_{f y}^{k, t} := \grad_{y} f(x_{k}, y_{k, t} ; \zeta_{y}^{k, t}), \\
h_{g y}^{k, t} & := \grad_{y} g(x_{k}, y_{k, t} ; \phi_{y}^{k, t}), \quad h_{g x y}^{k} := \grad_{x} g(x_{k}, y_{k+1} ; \phi_{x y}^{k}), \\
h_{f x}^{k} & := \grad_{x} f(x_{k}, y_{k+1} ; \zeta_{x}^{k}), \quad h_{g x z}^{k} := \grad_{x} g(x_{k}, z_{k+1} ; \phi_{x z}^{k}).
\end{aligned}
\end{equation}
With \(T=1\) and an appropriate choice of \(\xi\), Algorithm \(\mathrm{RF^2SA}\) enables a fully single-loop update of all variables, tailored to the manifold's geometry. The step-size design for \(\mathrm{RF^2SA}\), as outlined in Algorithm \ref{alg:RF2SA}, adapts to this setup.

\subsection{Step-Size Design Principle}
We tailor the step-sizes for Algorithm~\ref{alg:RF2SA} to ensure convergence to an $\epsilon$-stationary point of $F$. This involves meeting several geometric conditions, considering the curvature. For instance, if $\mathcal{L}_{\lambda_{k}}$ is $\left(\lambda_{k} \mu_{g} / 2\right)$-strongly convex along $y$'s geodesics, then updating $y_{k, t}$ resembles a geodesic contraction towards $y_{\lambda, k}^{*}$, with a rate of $1 - O(\mu_{g} \beta_{k})$.

Here, $\beta_{k} = \alpha_{k} \lambda_{k}$ is the effective step-size for $y_{k}$. We simplify notation by denoting $y_{\lambda, k}^{*} = \Exp_{x_k}^{-1}(y_{\lambda_{k}}^{*}(x_{k}))$ and $y_{k}^{*} = \Exp_{x_k}^{-1}(y^{*}(x_{k}))$, where $\Exp_{x_k}^{-1}(\cdot)$ represents the unique inverse of the exponential map at $x_k$, mapping points in the manifold back to the tangent space at $x_k$.

For updating $x_{k}$, the step-size $\xi \alpha_{k}$ should decay no slower than $\Omega(1/k)$. The step-size $\beta_{k}$ is limited to $O(1/l_{g, 1})$, implying a polynomial growth in $\lambda_{k}$ with $k$.

The manifold distance $\dist(\cdot, \cdot)$ between $x_{k+1}$ and $x_{k}$ depends on several factors. Ideally, $\lambda_{k}$'s growth rate ensures $\dist(y_{k}, y_{\lambda, k}^{*})$ is roughly $\lambda_{k}^{-2}$, suggesting $\lambda_{k}$ grows inversely to $\beta_k^{1/4}$.

Efficiency in Algorithm~\ref{alg:RF2SA} relies on how quickly $y_{k}$ and $z_{k}$ can track their targets as $x_{k}$ and $\lambda_{k}$ evolve. We will explore how $y_{\lambda}^{*}(x)$ adapts to changes in $\lambda$ and $x$.

\begin{figure}
\centering
\includegraphics[width=0.5\textwidth]{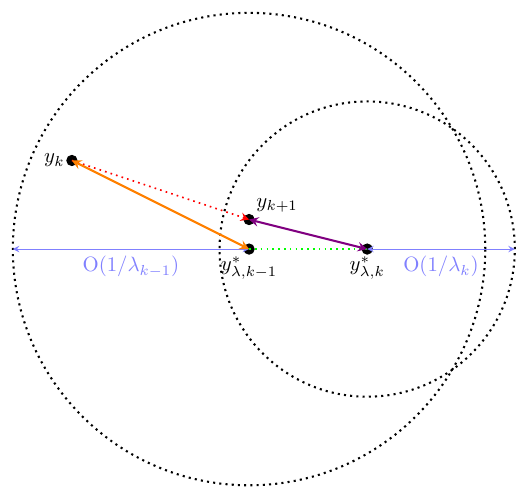}
\caption{$y_{k}$ should move faster than $y_{\lambda_{k}}^{*}(x_{k})$, remaining within an $O(1/\lambda_{k})$-ball around $y_{\lambda_{k}}^{*}(x_{k})$.}
\end{figure}

\begin{lemma}
\label{lemma:distance_bound}
For any points $x_1, x_2$ on a manifold $X$ and multipliers $\lambda_2 \geq \lambda_1 \geq 2 l_{f, 1} / \mu_g$, the distance between optimal solutions for these multipliers is bounded by
\begin{equation}
\dist(y_{\lambda_1}^*(x_1), y_{\lambda_2}^*(x_2)) \leq \frac{2(\lambda_2 - \lambda_1)}{\lambda_1 \lambda_2} \frac{l_{f, 0}}{\mu_g} + l_{\lambda, 0} \dist(x_2, x_1),
\end{equation}
with some constant $l_{\lambda, 0} \leq 3 l_{g, 1} / \mu_g$.
\end{lemma}

In the algorithmic setting, it's crucial that $y_k$'s update moves it sufficiently close to the current target $y_{\lambda, k}^*$ each iteration, surpassing the target's movement due to updates in $x_k$ and $\lambda_k$. Ideally, in expectation,
\begin{equation}
\dist(y_{k+1}, y_{\lambda, k}^*) < \dist(y_k, y_{\lambda, k-1}^*).
\end{equation}
Considering the geometry, the squared distance $\dist(y_{k+1}, y_{\lambda, k}^*)^2$ contracts with $T$-steps of $1 - O(\mu_g \beta_k)$, starting from $y_k$, leading to the requirement
\begin{equation}
(1 - O(T \mu_g \beta_k)) \dist(y_k, y_{\lambda, k}^*)^2 < \dist(y_k, y_{\lambda, k-1}^*)^2.
\end{equation}
Utilizing Lemma \ref{lemma:distance_bound} and the geometry, the minimal condition is
\begin{equation}
\begin{aligned}
\dist(y_{\lambda, k-1}, y_{\lambda, k}) &\leq \left(\frac{l_{f, 0}}{\mu_g}\right) \left(\frac{\delta_k}{\lambda_k^2}\right) + l_{\lambda, 0} \dist(x_k, x_{k-1}) \
&\leq T \mu_g \beta_k \dist(y_k, y_{\lambda, k-1}^*).
\end{aligned}
\end{equation}
The rate at which $\dist(y_{k+1}, y_{\lambda, k}^*)$ decreases must surpass $\lambda_k^{-1}$, maintaining controlled bias in $x_k$ updates. Additionally, $\dist(x_k, x_{k-1})$ should align with $\xi \beta_k \dist(y_k, y_{\lambda, k-1}^*)$.

Two key conditions emerge:
\begin{equation}
\frac{\delta_k}{\lambda_k} \leq O_P(1) \cdot \beta_k, \quad \frac{\xi}{T} < O_P(1),
\end{equation}
where $O_P(1)$ denotes constants dependent on the problem instance. If $\lambda_k$ increases polynomially, then $\delta_k / \lambda_k = O(1/k)$, satisfying the first condition if $\beta_k = \Omega(1/k)$. The second condition concerns the inner iterations $T$ needed per outer iteration, allowing for a single-loop algorithm with $T=1$ and an adequately small $\xi$, or setting $\xi=1$ and adjusting $T>1$ for specific instance parameters.

\section{Non-Asymptotic Convergence Analysis}

We discussed a number of assumptions on the regularity of the optimization objective and the underlying manifold in Section \ref{ss:assumptions}. We now present our main convergence result in Theorem \ref{thm:main}. Corollary \ref{cor:convergence-rates} provides explicit iteration complexity bounds for each setting, further elucidating the efficiency and applicability of our results in various contexts.

\subsection*{Convergence Analysis Results}

\label{ss:thm}
\begin{theorem}[Alexandrov Space Version]
\label{thm:main}
Given that the assumptions from Section~\ref{ss:assumptions} hold within an Alexandrov space with curvature bounded by \(\kappa\), analogous to manifolds \(\Mc, \Nc\), and assuming appropriate selection of parameters and step-sizes such that \(\lambda_{0} \geq 2 l_{f, 1} / \mu_{g}\) and

\begin{subequations}
\renewcommand{\theequation}{\theparentequation\alph{equation}} 
    \begin{align}
        \beta_{k} &\leq \gamma_{k} \leq \min \left(\frac{1}{4 l_{g, 1}}, \frac{1}{4 T \mu_{g}}\right),
        \alpha_{k} \leq \min \left(\frac{1}{8 l_{f, 1}}, \frac{1}{2 \xi l_{F, 1}}\right), \label{thm2.1}
    \end{align}
    \begin{align}
        \frac{\xi}{T} &< c_{\xi} \mu_{g} \cdot \max \left(l_{g, 1} l_{*, 0}^{2}, l_{*, 1} \sqrt{\Mc}\right)^{-1}, \frac{\delta_{k}}{\lambda_{k}} \leq \frac{T \mu_{g} \beta_{k}}{16}. \label{thm2.2}
    \end{align}
\end{subequations}

for all \(k \geq 0\), with \(c_{\xi}\) being a suitably chosen constant. Then, over \(K \geq 1\) iterations within the Alexandrov space—a generalization allowing curvature bounds without necessitating smoothness—the outcomes for Algorithm~\ref{alg:RF2SA} adhere to
\begin{equation}
\sum_{k=0}^{K-1} \xi \alpha_{k} \mathbb{E}\left[\left\| \grad F\left(x_{k}\right)\right\|_{x_{k}}^{2}\right] \leq O_{P}(1) \cdot \sum_{k} \xi \alpha_{k} \lambda_{k}^{-2} + O_{P}\left(\sigma_{f}^{2}\right) \cdot \sum_{k} \alpha_{k}^{2} \lambda_{k} + O_{P}\left(\sigma_{g}^{2}\right) \cdot \sum_{k} \gamma_{k}^{2} \lambda_{k} + O_{P}(1),
\label{eq:main-inequality}
\end{equation}
where \(\grad F(x)\) symbolizes a generalized notion of gradient in Alexandrov spaces, and \(\|\cdot\|_{x_{k}}\) signifies the distance measure at point \(x_{k}\), aligning with the space's metric structure.
\end{theorem}

The proof of Theorem~\ref{thm:main} is deferred to Appendix B. Therein we also explain why the effect of the curvature $\kappa$ is negligible on the final inequality \ref{eq:main-inequality}.

Our analysis examines the expected decrease of the potential function $\mathbb{V}_{k}$, defined as
\begin{equation}
\mathbb{V}_{k} := \left(F\left(x_{k}\right) - F^{*}\right) + l_{g, 1} \lambda_{k} \dist\left(y_{k}, y_{\lambda_{k}}^{*}\left(x_{k}\right)\right)^{2} + \frac{\lambda_{k} l_{g, 1}}{2} \dist\left(z_{k}, y^{*}\left(x_{k}\right)\right)^{2},\label{eq:4}
\end{equation}
where \(F^{*}\) is the minimum value of \(F\), and \(y_{\lambda}^{*}\) and \(y^{*}\) correspond to solutions. Monitoring the distance between \(y_{k}\) and \(y_{\lambda_{k}}^{*}(x_{k})\) is essential for computing the true gradient of \(F\) at \(x_{k}\) using only gradients. The proof will also show that the correct scaling factor for these errors is proportional to \(\lambda_{k}\).

For step-size design, we maintain conditions similar to \ref{thm2.1} for gradient-based methods. The conditions \ref{thm2.2} address the double-loop nature of the problem. Aligning with the step-size design rule (3), we propose:
\begin{equation}
\begin{split}
T &= \max \left(32, \left(c_{\xi} \mu_{g}\right)^{-1} \max \left(l_{g, 1} l_{*, 0}^{2}, \sqrt{M} l_{*, 1}\right)\right), \\
\xi &= 1, \quad \alpha_{k} = \frac{c_{\alpha}}{(k + k_{0})^{a}}, \quad \gamma_{k} = \frac{c_{\gamma}}{(k + k_{0})^{c}},
\end{split}
\label{eq:hyperparameters}
\end{equation}
and for the Lagrange multiplier increase sequence \(\{\delta_{k}\}\),
\begin{equation}
\delta_{k} = \min \left(\frac{T \mu_{g}}{16} \alpha_{k} \lambda_{k}^{2}, \frac{\gamma_{k}}{2 \alpha_{k}} - \lambda_{k}\right).
\label{eq:delta_k}
\end{equation}
Rate constants \(a, c \in [0,1]\) with \(a \geq c\), the initial value of the Lagrange multiplier \(\lambda_{0}\), and constants for the context are established as:
\begin{equation}
\begin{split}
k_{0} \geq \frac{4}{\mu_{g}} \max \left(\frac{\xi l_{F, 1}}{2}, T l_{g, 1}, l_{f, 1}\right), \lambda_{0} \geq \frac{2 l_{f, 1}}{\mu_{g}}, \\
c_{\gamma} = \frac{1}{\mu_{g} k_{0}^{1-c}}, \quad c_{\alpha} = \frac{1}{2 \lambda_{0} \mu_{g} k_{0}^{1-a}}.
\end{split}
\label{eq:constants}
\end{equation}
These specifications streamline convergence rate analysis, with the framework accommodating various other choices that comply with conditions \ref{thm2.1} and \ref{thm2.2}, facilitating the delineation of the convergence rate across different stochastic noise regimes.

In the following corollary, we present a more interpretable version of Theorem \ref{thm:main}, in terms of the iteration complexity guarantees under different settings. This interpretation allows for a clearer understanding of how our theoretical findings translate into practical implications for convergence rates in various scenarios.

\begin{corollary}\label{cor:convergence-rates}
Assume the stipulations of Theorem~\ref{thm:main} are upheld, with step-sizes delineated as in equations \ref{eq:hyperparameters}, \ref{eq:delta_k}, and \ref{eq:constants}. Let \( R \) signify a random variable uniformly distributed over \( \{0, \ldots, K-1\} \). Under these premises, after \( K \) iterations, the ensuing convergence outcomes are derived:

\begin{enumerate}[label=(\alph*)]
    \item\label{cor:conv-rates-a} In the presence of stochastic noise in both objectives \( f \) and \( g \) (\( \sigma_{f}^{2}, \sigma_{g}^{2} > 0 \)), setting \( a = 5/7 \) and \( c = 4/7 \), we achieve a convergence rate of \( \mathbb{E}\left[\| \grad F(x_{R}) \|^{2}\right] \asymp \frac{\log K}{K^{2/7}} \).
    
    \item\label{cor:conv-rates-b} If stochastic noise is solely in \( f \) (\( \sigma_{f}^{2} > 0 \), \( \sigma_{g}^{2} = 0 \)), setting \( a = 3/5 \) and \( c = 2/5 \), we attain \( \mathbb{E}\left[\| \grad F(x_{R}) \|^{2}\right] \asymp \frac{\log K}{K^{2/5}} \).
    
    \item\label{cor:conv-rates-c} In scenarios with exact gradients (\( \sigma_{f}^{2} = \sigma_{g}^{2} = 0 \)), appointing \( a = 1/3 \) and \( c = 0 \), it follows that \( \| \grad F(x_{K}) \|^{2} \asymp \frac{\log K}{K^{2/3}} \).
\end{enumerate}

These findings illustrate that convergence rates improve when stochastic noise affects fewer components of the problem. Specifically, the rate improves from \( O(k^{-2/7}) \) to \( O(k^{-2/5}) \) with noise only in \( f \), and to \( O(k^{-2/3}) \) in fully deterministic contexts. This compares to the \( O(k^{-1}) \) rates that can be obtained by second-order methods as in \cite{li2024riemannian, han2024framework}.
\end{corollary}

\section{Limitations and Future Work}

In general, fully first-order stochastic algorithms, although competitive with their second-order counterparts exhibit certain limitations like higher iteration complexity \citep{kwon2023fully}. This gap highlights the need for further investigation into the theoretical limits of first-order methods with respect to second-order methods for Riemannian bilevel optimization. Additionally, \algo's application is predominantly restricted to well-conditioned lower-level problems. It remains open to study \algo's potential in a broader range of problem classes.

Future research directions include utilizing our framework to address a wide array of real-world applications, where the natural settings of problems involve varying geometric structures at different decision levels. An exciting future direction is to explore applications of our algorithm in hyperparameter optimization, meta-learning, and reinforcement learning, where manifold optimization ideas have provided significant advantages. \citep{jaquier2020high, tabealhojeh2023rmaml, xu2016manifold, jaquier2020bayesian}

Investigating two-player games with states represented as matrices or other manifold-valued objects could further enrich the bilevel optimization landscape, offering novel insights into game theory and decision-making processes on complex geometries. Moreover, the integration of operator-valued optimization tasks and the development of algorithms that consider the manifold's curvature effects more explicitly would refine our understanding and application of Riemannian optimization techniques. \citep{domingo2020mean, cai2023curvature}

Moreover, in general, the concept of ``bilevel" formulations is becoming increasingly significant in the context of Riemannian problems, where many issues seem to naturally incorporate a two-tier optimization process. One pertinent example is the \(k\)-sparse barycenter problem (e.g., in \cite{do2023approximation}), where the goal is to approximate a covariance matrix \( X \) (represented as an ellipsoid) using a sparse combination of given covariance matrices \( A_1, \ldots, A_N \). Specifically, one aims to find a \(k\)-sparse weight vector \( q(X) \) that minimizes the distance between \( X \) and the Wasserstein barycenter of the selected matrices. Formally, \( q(X) := \arg\min_{q \in \mathcal{Q}} \text{dist}^2(X, \text{BaryCenter}(q, A_1, \ldots, A_N)) \), where \(\mathcal{Q} = \{ q \in \mathbb{R}^N \mid q \geq 0, \|q\|_0 \leq k, \sum_{i=1}^N q_i = 1 \}\). The barycenter is computed as \(\text{BaryCenter}(q, A_1, \ldots, A_N) = \arg\min_{Y \in \mathcal{M}} \sum_{i=1}^N q_i \text{dist}^2(Y, A_i)\), where \(\mathcal{M}\) is the manifold of symmetric positive definite matrices. Consequently, \( X \approx \text{BaryCenter}(q(X), A_1, \ldots, A_N) \), with \( X \) approximated using at most \( k \) of the covariance matrices \( A_1, \ldots, A_N \).

\section{Conclusion}

We have presented a novel and fully first-order approach to Riemannian bilevel optimization. This opens new avenues for addressing non-strongly convex lower-level problems and provides a geometrically aware framework for complex optimization challenges involving manifold constraints.

\bibliographystyle{plainnat}
\setlength{\bibsep}{3pt}
\bibliography{custom}


\clearpage

\appendix

\section*{Appendix / Supplemental Material}

\begin{table}[H]
    \centering
    \begin{tabular}{c l l}
        \toprule
        \textbf{Symbol} & \textbf{Meaning} & \textbf{Less than} \\
        \midrule
        $l_{f,0}$ & Bound of $\|\nabla_x f\|, \|\nabla_y f\|$ & $\cdot$ \\
        $l_{f,1}$ & Smoothness of $f$ & $\cdot$ \\
        $l_{g,0}$ & Bound of $\|\nabla_x g\|$ & $\cdot$ \\
        $l_{g,1}$ & Smoothness of $g$ & $\cdot$ \\
        $\mu_g$ & Strong-convexity of $g$ & $\cdot$ \\
        $l_{g,2}$ & Hessian-continuity of $g$ & $\cdot$ \\
        $M_f$ & Second-order moment of $\nabla f(x,y;\zeta)$ & $l_{f,0}^2 + \sigma_f^2$ \\
        $M_g$ & Second-order moment of $\nabla g(x,y;\phi)$ & $l_{g,0}^2 + \sigma_g^2$ \\
        $l_{f,2}$ & Hessian-continuity of $f$ & $\cdot$ \\
        $l_{F,1}$ & Smoothness of $F(x)$ & $l_{*,0} \left(l_{f,1} + \frac{l_{g,1}^2}{\mu_g} + \frac{2l_{f,0}l_{g,1}l_{g,2}}{\mu_g^2} \right)$ \\
        $l_{\lambda,0}$ & Lipschitzness of $y_\lambda^*(x)$ (for all $\lambda \ge 2l_{f,1}/\mu_g$) & $\frac{3 l_{g,1}}{\mu_g}$ \\
        $l_{\lambda,1}$ & Smoothness of $y_\lambda^*(x)$ (for $\lambda \ge 2l_{f,1}/\mu_g$) & $32 (l_{g,2} + \lambda^{-1} \cdot l_{f,2}) \frac{l_{g,1}^2}{\mu_g^3}$ \\
        $l_{*,0}$ & $= 1 + \max_{\lambda \ge 2l_{f,1}/\mu_g}l_{\lambda,0}$ & $\cdot$ \\
        $l_{*,1}$ & $= \max_{\lambda \ge 2l_{f,1}/\mu_g} l_{\lambda, 1}$ & $\cdot$ \\
        \bottomrule
    \end{tabular}
    \caption{Meaning of Constants}
    \label{tab:constant_relations}
\end{table}

In Table \ref{tab:constant_relations}, we list the main symbols used in the following proofs, their interpretations, and the inequalities that they satisfy (where applicable).

To simplify the representation of variable movements, we define $q^{x}_{k}$, $q^{y}_{k,t}$, and $q^{z}_{k,t}$ as follows:
\begin{equation*}
\begin{aligned}
    q^{x}_{k} & := \grad_{x} f(x_{k}, y_{k+1}) + \lambda_{k} (\grad_{x} g(x_{k}, y_{k+1}) - \grad_{x} g(x_{k}, z_{k+1})), \\
    q^{y}_{k,t} & := \grad_{y} f(x_{k}, y_{k,t}) + \lambda_{k} \grad_{y} g(x_{k}, y_{k,t}), \\
    q^{z}_{k,t} & := \grad_{y} g(x_{k}, z_{k,t}).
\end{aligned}
\end{equation*}
These quantities represent the expected movements of $x_{k}$, $y_{k}^{(t)}$, and $z_{k}^{(t)}$ in the absence of stochastic noise in the gradient oracles.

\section{Detailed Proofs of Lemmas \ref{lemma:riemannian_gradient_bound} and \ref{lemma:distance_bound}}

\subsection{Lemma A.1}
\label{A.1}

This lemma establishes a bound on the difference between the gradient of a function at two points \(x_2\) and \(x_1\), taking into account the effects of parallel transport.

$F(x) = f(x, y^*(x))$ is $l_{F, 1}$-smooth where
$$l_{F, 1} \leq l_{*, 0}\left(l_{f, 1}+\frac{l_{g, 1}^{2}}{\mu_{g}}+\frac{2 l_{f, 0} l_{g, 1} l_{g, 2}}{\mu_{g}^{2}}\right).$$

\begin{proof}
We recall from equation \eqref{Fhypergradient} that the gradient \(\grad F(x)\) of the function F, defined in \eqref{eq:P}, is given by the expression:
\[
\grad F(x) = \grad_x f(x, y^*(x)) - \Hess_{xy} g(x, y^*(x)) \left( \Hess_{yy} g(x, y^*(x)) \right)^{-1} \grad_y f(x, y^*(x)),
\]
where \(y^*(x)\) denotes the solution to the inner-level optimization problem associated with \(x\).

The bound for the difference between the parallel transported gradient at \(x_2\) to \(x_1\) and the gradient at \(x_1\) is given by:
\[
\begin{aligned}
&\left\|\mathrm{PT}_{x_2 \to x_1}\grad F\left(x_{2}\right)-\grad F\left(x_{1}\right)\right\| \\
&\quad \leq \left(l_{f, 1}+\frac{l_{f, 0}}{\mu_{g}} l_{g, 2} + \frac{l_{g, 1}}{\mu_{g}} l_{g, 1}\right)\left(\left\|d_{\Mc}(x_{1},x_{2})\right\|+\left\|d_{\Nc}(y^{*}\left(x_{1}\right),y^{*}\left(x_{2}\right))\right\|\right) \\
&\quad\quad + l_{g, 1} l_{f, 0}\left\|\mathrm{PT}_{x_2 \to x_1}\left(\nabla_{yy}^{2} g\left(x_{2}, y^{*}\left(x_{2}\right)\right)^{-1}\right)\mathrm{PT}_{x_1 \to x_2}-\nabla_{yy}^{2} g\left(x_{1}, y^{*}\left(x_{1}\right)\right)^{-1}\right\|.
\end{aligned}
\]
To simplify this inequality, we employ the assumptions on the smoothness and strong convexity of the functions \(f\) and \(g\), alongside the triangle inequality. From our assumptions in Section \ref{ss:assumptions}, we require in particular:

\begin{enumerate}
    \item \textbf{Smoothness of \(f\):} \(f\) is \(l_{f, 1^{-}}\)-smooth, which provides a bound on the gradient differences of \(f\) at two points.
    \item \textbf{Smoothness of \(g\):} \(g\) is \(l_{g, 1}\)-smooth, enabling us to bound the gradient differences of \(g\).
    \item \textbf{Strong Convexity of \(g\):} The \(\mu_{g}\)-strong convexity of \(g\) facilitates relating the Hessian of \(g\) to its inverse, critical for bounding the differences in Hessian inverses.
    \item \textbf{Lipschitz Continuity of the Hessian of \(g\):} The \(l_{g, 2}\)-Lipschitz continuity of the Hessian of \(g\) aids in bounding the differences in Hessians at two points.
\end{enumerate}
\[
\begin{aligned}
&\left\|\mathrm{PT}_{x_2 \to x_1}\mathrm{grad} F(x_{2})-\mathrm{grad} F(x_{1})\right\| \\
\leq &\; ||\mathrm{PT}_{x_2 \to x_1} \mathrm{grad} f(x_2, y^*(x_2)) - \mathrm{grad} f(x_1, y^*(x_1))|| \\
&+ ||\mathrm{PT}_{x_2 \to x_1} \mathrm{grad} g(x_2, y^*(x_2)) - \mathrm{grad} g(x_1, y^*(x_1))|| \\
&\cdot \max_x ||\mathrm{Hess}_{yy}g(x, y^*(x))^{-1}|| \cdot \max_x ||\mathrm{grad}_y f(x, y^*(x))|| \\
&+ ||\mathrm{PT}_{x_2 \to x_1} (\mathrm{Hess}_{xy} g(x_2, y^*(x_2))\mathrm{PT}_{x_1 \to x_2} - \mathrm{Hess}_{xy} g(x_1, y^*(x_1))|| \\ 
&\cdot \max_x ||\mathrm{Hess}_{yy}g(x, y^*(x))^{-1}|| \cdot \max_x ||\mathrm{grad}_y f(x, y^*(x))|| \\
&+ \max_x||\mathrm{Hess}_{xy} g(x, y^*(x))|| \cdot \max_x||\mathrm{grad}_y f(x, y^*(x))|| \\
& \cdot ||\mathrm{PT}_{x_2 \to x_1}\mathrm{Hess}^{-1}_{yy}(x_2, y^*(x_2))\mathrm{PT}_{x_1 \to x_2} - \mathrm{Hess}^{-1}_{yy}(x_1, y^*(x_1))|| \\
\leq &\; (l_{f_1} + \frac{l_{f_0}}{\mu_g} l_{g_2} + \frac{l_{g_1}}{\mu_g}l_{f_1}) (d_{\Mc}(x_1, x_2) + d_{\Nc}(y^*(x_1), y^*(x_2))) \\
&+ \frac{l_{g_1} l_{f_0}}{\mu_g^2} l_{g, 2} (d_{\Mc}(x_1, x_2) + d_{\Nc}(y^*(x_1), y^*(x_2))).
\end{aligned}
\]
where to upper bound $||\mathrm{PT}_{x_2 \to x_1}\mathrm{Hess}^{-1}_{yy}(x_2, y^*(x_2))\mathrm{PT}_{x_1 \to x_2} - \mathrm{Hess}^{-1}_{yy}(x_1, y^*(x_1))||$ we used the following result on bounding the norm of the difference of the inverses of two matrices, \(A\) and \(B\), leveraging the Neumann series to express the inverse of a matrix in terms of its perturbation. Specifically, we have:
\[
\|A^{-1} - B^{-1}\| \leq \|\Delta A\| \cdot \|A^{-1}\| \cdot \|B^{-1}\|,
\]
where $\Delta A = A - B$, which is instrumental in establishing the final bound on the gradient difference.

The lemma concludes with:
\[
\begin{aligned}
l_{F, 1} & \leq l_{*, 0}\left(l_{f, 1}+\frac{l_{f, 0} l_{g, 2}+l_{g, 1}^{2}}{\mu_{g}}+\frac{l_{f, 0} l_{g, 1} l_{g, 2}}{\mu_{g}^{2}}\right) \\
& \leq l_{*, 0}\left(l_{f, 1}+\frac{l_{g, 1}^{2}}{\mu_{g}}+\frac{2 l_{f, 0} l_{g, 1} l_{g, 2}}{\mu_{g}^{2}}\right),
\end{aligned}
\]
where the last inequality utilizes the condition that \(l_{g, 1} / \mu_{g} \geq 1\).
\end{proof}

\subsection{Lemma A.2}
\label{A.2}

This lemma establishes a bound on the difference between the gradient of a function \(F(x)\) and the gradient of a Lagrangian \(\mathcal{L}_{\lambda}(x, y)\) with respect to \(x\), adjusted for the effects of parallel transport.

For any $x, y, \lambda$, the following holds:
\[
\begin{array}{r}
\left\|\mathrm{grad} F(x) - \mathrm{grad}_x \mathcal{L}_{\lambda}(x, y) + \mathrm{PT}_{y \to y^*} \mathrm{Hess}_{xy} g(x, y^*) \left(\mathrm{Hess}_{yy} g(x, y^*)\right)^{-1} \mathrm{PT}_{y^* \to y} \mathrm{grad}_y \mathcal{L}(x, y)\right\| \\
\quad \leq 2\left(\frac{l_{g, 1}}{\mu_{g}}\right)d_{\Mc}(y, y^*)\left(l_{f, 1}+\lambda \cdot \min \left(2 l_{g, 1}, l_{g, 2}d_{\Nc}(y, y^*)\right)\right).
\end{array}
\]

\begin{proof}
Given the Lagrangian \(\mathcal{L}_{\lambda}(x, y)\), we consider the gradients with respect to variables \(x\) and \(y\), expressed as:
\[
\begin{aligned}
\grad_x \mathcal{L}_{\lambda}(x, y) &= \grad_x f(x, y) + \lambda \left( \grad_x g(x, y) - \text{PT}_{y^*(x) \to y} \grad_x g(x, y^*(x)) \right), \\
\grad_y \mathcal{L}_{\lambda}(x, y) &= \grad_y f(x, y) + \lambda \grad_y g(x, y).
\end{aligned}
\]
Here, \(\text{PT}_{y^*(x) \to y}\) denotes the parallel transport operation that moves vectors along geodesics from the tangent space at \(y^*(x)\) to the tangent space at \(y\), ensuring that the comparison of vectors is meaningful.

The discrepancy between the gradient of \(F\) and the gradient of the Lagrangian with respect to \(x\) is detailed as follows:
\[
\begin{aligned}
\grad F(x) - \grad_x \mathcal{L}_{\lambda}(x, y) = & \grad_x f(x, y^*) - \text{PT}_{y \to y^*} \grad_x f(x, y) \\
& - \text{Hess}_{xy} g(x, y^*) \left( \text{Hess}_{yy} g(x, y^*)^{-1} \grad_y f(x, y) \right) \\
& - \lambda \left( \grad_x g(x, y) - \text{PT}_{y^* \to y} \grad_x g(x, y^*) \right).
\end{aligned}
\]

We can rearrange terms for $\grad_x g(x, y) - \mathrm{PT}_{y^* \to y} \grad_x g(x, y^*)$ as the following:
\[
\begin{aligned}
\grad_x g(x, y) - \mathrm{PT}_{y^* \to y} \grad_x g(x, y^*) = & \grad_x g(x, y) - \mathrm{PT}_{y^* \to y} \grad_x g(x, y^*) \\
& - \mathrm{PT}_{y^* \to y} \text{Hess}_{xy} g(x, y^*) \mathrm{PT}_{y \to y^*} \text{Exp}_y^{-1}(y^*) \\
& + \mathrm{PT}_{y^* \to y} \text{Hess}_{xy} g(x, y^*) \mathrm{PT}_{y \to y^*} \text{Exp}_y^{-1}(y^*).
\end{aligned}
\]

Note that from the optimality condition for $y^*$, we have $\mathrm{grad}_y g(x, y^*) = 0$. From the gradient of the Lagrangian $\mathcal{L}$, we have $\mathrm{grad}_y \mathcal{L}(x, y) = \mathrm{grad}_y f(x, y) + \lambda \mathrm{grad}_y g(x, y)$. We can express the equivalent of $y - y^*$ using the inverse exponential map and the Hessian as follows:
\[
\begin{aligned}
\text{Exp}_y^{-1}(y^*) = & -\left(\text{Hess}_{yy} g(x, y^*)\right)^{-1} \left(\grad_y g(x, y) - \text{PT}_{y^* \to y} \grad_y g(x, y^*) \right. \\
& \left. - \text{PT}_{y^* \to y}\text{Hess}_{yy} g(x, y^*)\text{PT}_{y \to y^*} \text{Exp}_y^{-1}(y^*)\right) \\
& + \frac{1}{\lambda} \left(\text{Hess}_{yy} g(x, y^*)\right)^{-1} \left(\grad_y \mathcal{L}(x, y) - \grad_y f(x, y)\right).
\end{aligned}
\]

This approximation is based on the Taylor expansion in the setting, where \( \mathrm{grad}_y g(x, y^*) = 0 \) because \( y^* \) is an optimal point (assuming \( g \) is minimized at \( y^* \) with respect to \( y \)).
\[ \mathrm{grad}_y g(x, y) - \mathrm{PT}_{y^* \to y} \mathrm{grad}_y g(x, y^*) \approx \mathrm{PT}_{y^* \to y}\mathrm{Hess}_{yy} g(x, y^*)\mathrm{PT}_{y \to y^*} \mathrm{Exp}_y^{-1}(y) \]
\[
\begin{aligned}
\mathrm{grad} F(x) - \mathrm{grad}_x \mathcal{L}_{\lambda}(x, y) \\ 
& = \left(\mathrm{grad}_x f(x, y^*) - \mathrm{grad}_x f(x, y)\right) \\
& - \mathrm{Hess}_{xy} g(x, y^*) \left(\mathrm{Hess}_{yy} g(x, y^*)\right)^{-1} \left(\mathrm{grad}_y f(x, y^*) - \mathrm{PT}_{y \to y^*} \mathrm{grad}_y f(x, y)\right) \\
& - \mathrm{Hess}_{xy} g(x, y^*) \left(\mathrm{Hess}_{yy} g(x, y^*)\right)^{-1} \mathrm{PT}_{y \to y^*} \mathrm{grad}_y \mathcal{L}(x, y) \\
& - \lambda \left(\mathrm{grad}_x g(x, y) - \mathrm{PT}_{y^* \to y} \mathrm{grad}_x g(x, y^*) - \mathrm{PT}_{y^* \to y} \mathrm{Hess}_{xy} g(x, y^*) \mathrm{Exp}_y^{-1}(y^*)\right) \\
& + \lambda \mathrm{Hess}_{xy} g(x, y^*) \left(\mathrm{Hess}_{yy} g(x, y^*)\right)^{-1} \left(\mathrm{grad}_y g(x, y) - \mathrm{PT}_{y^* \to y} \mathrm{grad}_y g(x, y^*) \right. \\
& \left. - \mathrm{PT}_{y^* \to y} \mathrm{Hess}_{yy} g(x, y^*) \mathrm{PT}_{y \to y^*} \mathrm{Exp}_y^{-1}(y^*)\right).
\end{aligned}
\]

To simplify this, we will require the following facts:
\[
\left\|\mathrm{grad}_y g(x, y) - \mathrm{PT}_{y^* \to y} \mathrm{grad}_y g(x, y^*) - \mathrm{PT}_{y^* \to y} \mathrm{Hess}_{yy} g(x, y^*)\mathrm{PT}_{y \to y^*}\mathrm{Exp}_{y^*}^{-1}(y)\right\| \leq l_{g, 2} \|\mathrm{Exp}_{y^*}^{-1}(y)\|^2
\]
\[
\left\|\mathrm{grad}_y g(x, y) - \mathrm{PT}_{y^* \to y} \mathrm{grad}_y g(x, y^*) - \mathrm{PT}_{y^* \to y} \mathrm{Hess}_{yy} g(x, y^*)\mathrm{PT}_{y \to y^*}\mathrm{Exp}_{y^*}^{-1}(y)\right\| \leq 2 l_{g, 1} d_{\Mc}(y, y^*)
\]
\[
\left\|\mathrm{grad}_x g(x, y) - \mathrm{PT}_{y^* \to y} \mathrm{grad}_x g(x, y^*) - \mathrm{PT}_{y^* \to y} \mathrm{Hess}_{xy} g(x, y^*)\mathrm{PT}_{y \to y^*}\mathrm{Exp}_{y^*}^{-1}(y)\right\| \] \[\leq \min \left(l_{g, 2} d_{\Mc}(y, y^*)^2, 2 l_{g, 1} d_{\Mc}(y, y^*)\right) .
\]
\[
\left\|\mathrm{grad}_x f(x, y^*) - \mathrm{PT}_{y \to y^*} \mathrm{grad}_x f(x, y)\right\| \leq l_{f, 1} d_{\Nc}(y, y^*),
\]
\[
\left\|\mathrm{grad}_y f(x, y^*) - \mathrm{PT}_{y \to y^*} \mathrm{grad}_y f(x, y)\right\| \leq l_{f, 1} d_{\Nc}(y, y^*).
\]

With this, our final result is:
\[
\begin{aligned}
& \left\|\mathrm{grad} F(x) - \mathrm{grad}_x \mathcal{L}_{\lambda}(x, y) + \mathrm{PT}_{y \to y^*} \mathrm{Hess}_{xy} g(x, y^*) \left(\mathrm{Hess}_{yy} g(x, y^*)\right)^{-1} \mathrm{PT}_{y^* \to y} \mathrm{grad}_y \mathcal{L}(x, y)\right\| \\
& \quad \leq l_{f, 1}\left(1+\frac{l_{g, 1}}{\mu_{g}}\right)d_{\Nc}(y, y^*) + \lambda\left(1+\frac{l_{g, 1}}{\mu_{g}}\right)d_{\Nc}(y, y^*) \min \left(l_{g, 2}d_{\Mc}(y, y^*)^2, 2 l_{g, 1}\right).
\end{aligned}
\]

We know that $l_{g, 1} / \mu_{g} \geq 1$ and thus, we have:
\[
\begin{array}{r}
\left\|\mathrm{grad} F(x) - \mathrm{grad}_x \mathcal{L}_{\lambda}(x, y) + \mathrm{PT}_{y \to y^*} \mathrm{Hess}_{xy} g(x, y^*) \left(\mathrm{Hess}_{yy} g(x, y^*)\right)^{-1} \mathrm{PT}_{y^* \to y} \mathrm{grad}_y \mathcal{L}(x, y)\right\| \\
\quad \leq 2\left(\frac{l_{g, 1}}{\mu_{g}}\right)d_{\Nc}(y, y^*)\left(l_{f, 1}+\lambda \cdot \min \left(2 l_{g, 1}, l_{g, 2}d_{\Nc}(y, y^*)\right)\right).
\end{array}
\]
\end{proof}

\subsection{Lemma A.3}
\label{A.3}

Under Assumptions \ref{assumption:objective_functions}, \ref{assumption:lower_level} and \ref{assumption:gradient_access}, and $\lambda > 2l_{f,1} / \mu_g$, a function $y_{\lambda}^*(x)$ is $l_{\lambda,1}$-smooth: for any $x_1, x_2 \in X$, we have $$\|\grad y_\lambda^*(x_1) - \grad y_\lambda^*(x_2)\| \le l_{\lambda, 1} d_{\Mc}(x_1, x_2)$$ where $l_{\lambda,1} \le 32 (l_{g,2} + \lambda^{-1} l_{f,2}) l_{g,1}^2 / \mu_g^3$.

\begin{proof}
The Lipschitz continuity of $y_{\lambda}^{*}(x)$ directly follows from Lemma \ref{lemma:distance_bound}, considering the manifold's intrinsic geometry. By the optimality condition for $\nabla y_{\lambda}^{*}(x)$ in the manifold setting, we obtain
\[
\nabla_{y} \mathcal{L}_{\lambda}\left(x, y_{\lambda}^{*}(x)\right)=\nabla_{y} f\left(x, y_{\lambda}^{*}(x)\right)+\lambda \nabla_{y} g\left(x, y_{\lambda}^{*}(x)\right)=0 ,
\]

where $\nabla_{y}$ denotes the gradient with respect to $y$. Differentiating with respect to $x$ along the manifold yields
\[
\left(\text{Hess}_{yy} f\left(x, y_{\lambda}^{*}(x)\right)+\lambda \text{Hess}_{yy} g\left(x, y_{\lambda}^{*}(x)\right)\right) \nabla y_{\lambda}^{*}(x)=-\left(\text{Hess}_{xy} f\left(x, y_{\lambda}^{*}(x)\right)+\lambda \text{Hess}_{xy} g\left(x, y_{\lambda}^{*}(x)\right)\right),
\]

where $\text{Hess}_{yy}$ and $\text{Hess}_{xy}$ represent the Hessians with respect to $y$ and the mixed partial derivative along the manifold, respectively. Given $\lambda>2 l_{f, 1} / \mu_{g}$, the left-hand side exhibits a positive definiteness with a minimum eigenvalue greater than $\lambda \mu_{g} / 2$. Thus,

\[
\nabla y_{\lambda}^{*}(x)=-\left(\frac{1}{\lambda} \text{Hess}_{yy} f\left(x, y_{\lambda}^{*}(x)\right)+\text{Hess}_{yy} g\left(x, y_{\lambda}^{*}(x)\right)\right)^{-1}\left(\frac{1}{\lambda} \text{Hess}_{xy} f\left(x, y_{\lambda}^{*}(x)\right)+\text{Hess}_{xy} g\left(x, y_{\lambda}^{*}(x)\right)\right) .
\]

To derive the smoothness property, we compare the expression at two points $x_{1}$ and $x_{2}$ on the manifold:
\[
\begin{aligned}
\frac{\lambda \mu_{g}}{2}\|\nabla y_{\lambda}^{*}(x_{1})-\nabla y_{\lambda}^{*}(x_{2})\| \leq & \left(l_{f, 2}+\lambda l_{g, 2}\right)\left(d_{\Mc}(x_{1},x_{2})+d_{\Nc}(y_{\lambda}^{*}(x_{1}),y_{\lambda}^{*}(x_{2}))\right) \max _{x \in X}\|\nabla y_{\lambda}^{*}(x)\| \\
& +\left(l_{f, 2}+\lambda l_{g, 2}\right)\left(d_{\Mc}(x_{1},x_{2})+d_{\Nc}(y_{\lambda}^{*}(x_{1}),y_{\lambda}^{*}(x_{2}))\right) \\
\leq & \left(l_{f, 2}+\lambda l_{g, 2}\right)\left(1+l_{\lambda, 0}\right)^{2}d_{\Mc}(x_{1},x_{2}) .
\end{aligned}
\]

Rearranging, we obtain

\[
\|\nabla y_{\lambda}^{*}(x_{1})-\nabla y_{\lambda}^{*}(x_{2})\| \leq 32\left(\frac{l_{f, 2}}{\lambda}+l_{g, 2}\right) \frac{l_{g, 1}^{2}}{\mu_{g}^{3}}d_{\Mc}(x_{1},x_{2}).
\]
\end{proof}

\subsection{Lemma A.4}
\label{A.4}

For any fixed $\lambda > 2l_{f,1} / \mu_g$, at every $k$ iteration conditioned on $\mathcal{F}_k$, we have
\begin{align*}
    \Exs[\|d_{\Nc}\left(y^{*}\left(x_{k+1}\right), y^{*}\left(x_{k}\right)\right)^2| \mathcal{F}_k] \le \xi^2 l_{*,0}^2 \left( \alpha_k^2 \Exs[\|q_k^x\|^2 | \mathcal{F}_k] + \alpha_k^2 \sigma_f^2 + \beta_k^2 \sigma_g^2 \right). 
\end{align*}

\begin{proof}
The result directly follows from the Lipschitz continuity established in Lemma \ref{lemma:distance_bound}, taking the limit as \(\lambda_{1}=\lambda_{2}\) approaches infinity on a manifold.

Given the structure of a manifold, we assess the changes in the optimal solution \(y^{*}\) between consecutive points \(x_{k+1}\) and \(x_{k}\) through the geodesic distance, conditioned on the filtration \(\mathcal{F}_{k}\). This approach quantifies the modifications in \(y^{*}\) as we traverse from one location to another on the manifold. Specifically, we express the expectation of the squared geodesic distance between \(y^{*}(x_{k+1})\) and \(y^{*}(x_{k})\) as follows:

$$
\mathbb{E}\left[d_{\Nc}\left(y^{*}\left(x_{k+1}\right), y^{*}\left(x_{k}\right)\right)^2 \mid \mathcal{F}_{k}\right] \leq l_{*, 0}^2 \mathbb{E}\left[d_{\Mc}\left(x_{k+1}, x_{k}\right)^2 \mid \mathcal{F}_{k}\right],
$$

The inequality captures the bounded change in \(y^{*}\) in response to movements in \(x\) across the manifold, leveraging the Lipschitz property of \(y^{*}\) relative to \(x\).

Further, by incorporating the step-sizes and the stochastic gradients' variances, we refine this inequality to:

$$
\leq l_{*, 0}^2 \xi^2 \alpha_{k}^2\left(\mathbb{E}\left[\left\|q_{k}^{x}\right\|^2 \mid \mathcal{F}_{k}\right]+\alpha_{k}^2 \sigma_{f}^2+\beta_{k}^2 \sigma_{g}^2\right),
$$

where \(\xi^2\), \(\alpha_{k}\), \(\sigma_{f}\), and \(\sigma_{g}\) encapsulate the effect of the algorithm's parameters and the inherent randomness of the optimization problem within the manifold setting. The expression \(\mathbb{E}\left[\left\|q_{k}^{x}\right\|^2 \mid \mathcal{F}_{k}\right]\) reflects the expected squared norm of the search direction on the tangent space.
\end{proof}

\subsection{Lemma A.5}
\label{A.5}

At every $k^{th}$ iteration, conditioned on $\mathcal{F}_k$, let $v_k$ be a random vector decided before updating $x_k$. Then for any $\eta_k > 0$, we have
\begin{align*}
    \Exs[\langle v_k, {y^*(x_{k+1}) - y^*(x_k)}\rangle | F_k ] &\le (\xi \alpha_k \eta_k + M \xi^2 l_{*,1}^2 \beta_k^2) \Exs[\|v_k\|^2 | F_k] \\
    &\quad + \left(\frac{\xi \alpha_k l_{*,0}^2}{4\eta_k} + \frac{\xi^2 \alpha_k^2}{4}\right) \Exs[\|q_k^x\|^2 | \mathcal{F}_k] + \frac{\xi^2}{4}(\alpha_k^2 \sigma_f^2 + \beta_k^2 \sigma_g^2),
\end{align*}
where $M := \max\left(l_{f,0}^2 + \sigma_f^2, l_{g,0}^2 + \sigma_g^2\right)$.

\begin{proof}
Utilizing the smoothness property of $y^{*}(x)$ as discussed in \cite{chen2021closing}, which is essential for controlling the noise variance induced by updating $x$, we proceed as follows on a manifold:

Consider the inner product on the tangent space of the manifold at point $x_k$, which respects the manifold's geometry. For two vectors $u, v$ in the tangent space at $x_k$, their inner product is denoted by $\langle u, v \rangle_{x_k}$. We can then express the expectation involving this inner product as follows:

$$
\begin{aligned}
\langle v_{k}, \Exp_{x_k}^{-1}(y_{k+1}^{*}) - \Exp_{x_k}^{-1}(y_{k}^{*}) \rangle_{x_k} = & \langle v_{k}, \nabla y^{*}(x_{k})(\Exp_{x_k}^{-1}(x_{k+1})) \rangle_{x_k} \\
& + \langle v_{k}, \Exp_{x_k}^{-1}(y^{*}(x_{k+1})) - \Exp_{x_k}^{-1}(y^{*}(x_{k})) - \nabla y^{*}(x_{k})(\Exp_{x_k}^{-1}(x_{k+1})) \rangle_{x_k}.
\end{aligned}
$$

For the first term, applying the expectation and the Cauchy-Schwarz inequality on the manifold, we get:

$$
\mathbb{E}[\langle v_{k}, \nabla y^{*}(x_{k})(\Exp_{x_k}^{-1}(x_{k+1})) \rangle_{x_k} \mid \mathcal{F}_{k}] = -\xi \alpha_{k} \mathbb{E}[\langle v_{k}, \nabla y^{*}(x_{k}) q_{k}^{x} \rangle_{x_k} \mid \mathcal{F}_{k}]
$$

$$
\leq \xi \alpha_{k} \eta_{k} \mathbb{E}[\|v_{k}\|^2_{x_k} \mid \mathcal{F}_{k}] + \frac{\xi \alpha_{k}}{4 \eta_{k}} \mathbb{E}[\|\nabla y^{*}(x_{k}) q_{k}^{x}\|^2_{x_k} \mid \mathcal{F}_{k}]
$$

$$
\leq \xi \alpha_{k} \eta_{k} \mathbb{E}[\|v_{k}\|^2_{x_k} \mid \mathcal{F}_{k}] + \frac{\xi \alpha_{k} l_{*, 0}^{2}}{4 \eta_{k}} \mathbb{E}[\|q_{k}^{x}\|^2_{x_k} \mid \mathcal{F}_{k}]
$$

For the second term, leveraging the smoothness of $y^{*}(x)$ on the manifold, we have:

$$
\mathbb{E}[\langle v_{k}, \Exp_{x_k}^{-1}(y^{*}(x_{k+1})) - \Exp_{x_k}^{-1}(y^{*}(x_{k})) - \nabla y^{*}(x_{k})(\Exp_{x_k}^{-1}(x_{k+1})) \rangle_{x_k} \mid \mathcal{F}_{k}]
$$

$$
\leq \frac{l_{*, 1}}{2} \mathbb{E}[d_{\Mc}(x_{k+1}, x_{k})^2 \mid \mathcal{F}_{k}]
$$
\end{proof}

\subsection{Lemma A.6}
\label{A.6}

Under Assumptions \ref{assumption:objective_functions}-\ref{assumption:f_second_order}, at every $k^{th}$ iteration, conditioned on $\mathcal{F}_k$, let $v_k$ be a random vector decided before updating $x_k$. Then for any $\eta_k > 0$, we have
\begin{align*}
    \Exs[\langle {v_k}, {y_{\lambda_{k+1}}^*(x_{k+1}) - y_{\lambda_k}^*(x_k)} \rangle | \mathcal{F}_k ] &\le (\delta_k/\lambda_k + \xi \alpha_k \eta_k + M \xi^2 l_{\lambda_k,1}^2 \beta_k^2) \Exs[\|v_k\|^2 | \mathcal{F}_k] \\
    &+ \left(\frac{\xi \alpha_k l_{*,0}^2}{4\eta_k} + \frac{\xi^2 \alpha_k^2}{4}\right) \Exs[\|q_k^x\|^2 | \mathcal{F}_k] + \frac{\xi^2}{4} (\alpha_k^2 \sigma_f^2 + \beta_k^2 \sigma_g^2) + \frac{\delta_k l_{f,0}^2}{\lambda_k^3 \mu_g^2},
    \end{align*}
where $M := \max\left(l_{f,0}^2 + \sigma_f^2, l_{g,0}^2 + \sigma_g^2\right)$.

\begin{proof}
We start with the following decomposition:

\[
\begin{aligned}
\langle v_{k}, & \Exp_{x_{k}}^{-1}(y_{\lambda_{k+1}}^{*}(x_{k+1})) - \Exp_{x_{k}}^{-1}(y_{\lambda_{k}}^{*}(x_{k})) \rangle = \\
& \langle v_{k}, \Exp_{x_{k+1}}^{-1}(y_{\lambda_{k+1}}^{*}(x_{k+1})) - \Exp_{x_{k+1}}^{-1}(y_{\lambda_{k}}^{*}(x_{k+1})) \rangle \\
& +\langle v_{k}, \nabla y_{\lambda_{k}}^{*}(x_{k})(\Exp_{x_{k}}^{-1}(x_{k+1})) \rangle \\
& +\langle v_{k}, \Exp_{x_{k}}^{-1}(y_{\lambda_{k}}^{*}(x_{k+1})) - \Exp_{x_{k}}^{-1}(y_{\lambda_{k}}^{*}(x_{k})) - \nabla y_{\lambda_{k}}^{*}(x_{k})(\Exp_{x_{k}}^{-1}(x_{k+1})) \rangle.
\end{aligned}
\]

For the second and third terms, the smoothness of $y_{\lambda}(x)$ is applied similarly to the proof in \ref{A.5}, considering the manifold's intrinsic geometry.

Regarding the first term, taking expectation and using the inequality $\langle a, b\rangle \leq c\|a\|^2 + \frac{1}{4c}\|b\|^2$ adapted for the tangent space, we get:
\[
\begin{aligned}
\mathbb{E}&[\langle v_{k}, \Exp_{x_{k+1}}^{-1}(y_{\lambda_{k+1}}^{*}(x_{k+1})) - \Exp_{x_{k+1}}^{-1}(y_{\lambda_{k}}^{*}(x_{k+1})) \rangle \mid \mathcal{F}_{k}] \\
& \leq c \mathbb{E}[\|v_{k}\|^2] + \frac{1}{4c} \mathbb{E}[\|\Exp_{x_{k+1}}^{-1}(y_{\lambda_{k+1}}^{*}(x_{k+1})) - \Exp_{x_{k+1}}^{-1}(y_{\lambda_{k}}^{*}(x_{k+1}))\|^2] \\
& \leq c \mathbb{E}[\|v_{k}\|^2] + \frac{1}{c} \frac{\delta_{k}^2}{\lambda_{k}^2 \lambda_{k+1}^2} \frac{l_{f,0}^2}{\mu_{g}^2},
\end{aligned}
\]
where the expectation and norms are understood within the context of the manifold's geometry. By selecting $c = \frac{\delta_{k}}{\lambda_{k}}$, we derive:
\[
\mathbb{E}[\langle v_{k}, \Exp_{x_{k+1}}^{-1}(y_{\lambda_{k+1}}^{*}(x_{k+1})) - \Exp_{x_{k+1}}^{-1}(y_{\lambda_{k}}^{*}(x_{k+1})) \rangle \mid \mathcal{F}_{k}] \leq \frac{\delta_{k}}{\lambda_{k}} \mathbb{E}[\|v_{k}\|^2] + \frac{l_{f,0}^2 \delta_{k}}{\mu_{g}^2 \lambda_{k}^3}.
\]
Combining this with bounds on the other two terms, we conclude the lemma.
\end{proof}

\subsection*{Proof of Lemma 1}
\label{Lemma 1}

Let \( y_{\lambda}^*(x) \coloneqq \arg \min_{y} \mathcal{L}_{\lambda}(x, y) \). Note that \(\mathrm{grad}_y \mathcal{L}_{\lambda}(x, y_{\lambda}^*(x)) = 0\), and thus
\[
\mathrm{grad} \mathcal{L}_{\lambda}^*(x) = \mathrm{grad}_x \mathcal{L}_{\lambda}(x, y_{\lambda}^*(x)) + \mathrm{grad}_x y_{\lambda}^*(x)^T \mathrm{grad}_y \mathcal{L}_{\lambda}(x, y_{\lambda}^*(x)) = \mathrm{grad}_x \mathcal{L}_{\lambda}(x, y_{\lambda}^*(x)).
\]

To compare this to \(\mathrm{grad} F(x)\), we can invoke Lemma \ref{A.2} which gives
\[
\begin{aligned}
& ||\mathrm{grad} F(x) - \mathrm{grad}_x \mathcal{L}_{\lambda}(x, y_{\lambda}^*(x))|| \\
& \quad \leq 2\left(l_{g, 1} / \mu_{g}\right) d_{\Nc}(y_{\lambda}^*(x), y^*(x)) \left(l_{f, 1} + \lambda \cdot \min \left(2 l_{g, 1}, l_{g, 2} d_{\Nc}(y^*(x), y_{\lambda}^*(x))\right)\right) .
\end{aligned}
\]

From a version of Lemma 2 (\ref{Lemma 2}), we use \(d_{\Nc}(y_{\lambda}^*(x), y^*(x)) \leq \frac{2 l_{f, 0}}{\lambda \mu_{g}}\), and get
\[
||\mathrm{grad} F(x) - \mathrm{grad}_x \mathcal{L}_{\lambda}(x, y_{\lambda}^*(x))|| \leq \frac{1}{\lambda} \cdot \frac{4 l_{f, 0} l_{g, 1}}{\mu_{g}^2} \left(l_{f, 1} + \frac{2 l_{f, 0} l_{g, 2}}{\mu_{g}}\right). \qed
\]

Here, \( \mathrm{grad}_x y_{\lambda}^*(x)^T \) represents the transpose of the gradient of \( y_{\lambda}^*(x) \) with respect to \( x \).

\subsection*{Proof of Lemma 2}
\label{Lemma 2}

Note that on a manifold, the function \(\mathcal{L}_{\lambda}(x, y)\) is at least \(\frac{\lambda \mu_{g}}{2}\) strongly-convex in \(y\) with respect to the metric once \(\lambda \geq 2 l_{f, 1} \mu_{g}\). To see this,
\[
\mathcal{L}_{\lambda}(x, y) = f(x, y) + \lambda (g(x, y) - g^*(x))
\]
which is at least \(-l_{f, 1} + \lambda \mu_{g}\)-strongly convex in \(y\) with respect to the metric. If \(\lambda > 2 l_{f, 1} / \mu_{g}\), this implies at least \(\lambda \mu_{g} / 2\) strong-convexity of \(\mathcal{L}_{\lambda}(x, y)\) in \(y\).

By the optimality condition at \(y_{\lambda_{1}}^*(x_{1})\) with \(x_{1}, \lambda_{1}\), we have
\[
\mathrm{grad}_y f(x_{1}, y_{\lambda_{1}}^*(x_{1})) + \lambda_{1} \mathrm{grad}_y g(x_{1}, y_{\lambda_{1}}^*(x_{1})) = 0,
\]

which also implies that \(d_M(g(x_{1}, y_{\lambda_{1}}^*(x_{1})), 0) \leq l_{f, 0} / \lambda_{1}\). Observe that
\[
\begin{aligned}
\mathrm{grad}_y f & (x_2, y_{\lambda_1}^*(x_1)) + \lambda_2 \mathrm{grad}_y g(x_2, y_{\lambda_1}^*(x_1)) \\
= & \left(\mathrm{grad}_y f(x_2, y_{\lambda_1}^*(x_1)) - \mathrm{PT}_{x_1 \rightarrow x_2}(\mathrm{grad}_y f(x_1, y_{\lambda_1}^*(x_1)))\right) + \mathrm{PT}_{x_1 \rightarrow x_2}(\mathrm{grad}_y f(x_1, y_{\lambda_1}^*(x_1))) \\
& + \lambda_2 \left(\mathrm{grad}_y g(x_2, y_{\lambda_1}^*(x_1)) - \mathrm{PT}_{x_1 \rightarrow x_2}(\mathrm{grad}_y g(x_1, y_{\lambda_1}^*(x_1)))\right) + \lambda_2 \mathrm{PT}_{x_1 \rightarrow x_2}(\mathrm{grad}_y g(x_1, y_{\lambda_1}^*(x_1))) \\
= & \left(\mathrm{grad}_y f(x_2, y_{\lambda_1}^*(x_1)) - \mathrm{PT}_{x_1 \rightarrow x_2}(\mathrm{grad}_y f(x_1, y_{\lambda_1}^*(x_1)))\right) \\
& + \lambda_2 \left(\mathrm{grad}_y g(x_2, y_{\lambda_1}^*(x_1)) - \mathrm{PT}_{x_1 \rightarrow x_2}(\mathrm{grad}_y g(x_1, y_{\lambda_1}^*(x_1)))\right) \\
& + (\lambda_2 - \lambda_1) \mathrm{PT}_{x_1 \rightarrow x_2}(\mathrm{grad}_y g(x_1, y_{\lambda_1}^*(x_1))),
\end{aligned}
\]

where in the last equality, we applied the optimality condition for \(y_{\lambda_{1}}^*(x_{1})\). Then applying the Lipschitzness of \(\mathrm{grad}_y f\) and \(\mathrm{grad}_y g\) in \(x\), we have
\[
d_{\Mc}(\mathrm{grad}_y f(x_{2}, y_{\lambda_{1}}^*(x_{1})) + \lambda_{2} \mathrm{grad}_y g(x_{2}, y_{\lambda_{1}}^*(x_{1})), 0) \leq l_{f, 1} d_{\Mc}(x_{1}, x_{2}) + l_{g, 1} \lambda_{2} d_{\Mc}(x_{2}, x_{1}) + (\lambda_{2} - \lambda_{1}) \frac{l_{f, 0}}{\lambda_{1}} .
\]

Since \(\mathcal{L}_{\lambda_{2}}(x_{2}, y)\) is \(\lambda_{2} \mu_{g} / 2\)-strongly convex in \(y\) with respect to the metric, from the coercivity property of strongly-convex functions, along with the optimality condition with \(y_{\lambda_{2}}^*(x_{2})\), we have
\[
\begin{aligned}
& \frac{\lambda_{2} \mu_{g}}{2} d_{\Mc}(y_{\lambda_{1}}^*(x_{1}), y_{\lambda_{2}}^*(x_{2})) \\
& \quad \leq ||\nabla_y \mathcal{L}_{\lambda_{2}}(x_{2}, y_{\lambda_{1}}^*(x_{1})) - \nabla_y \mathcal{L}_{\lambda_{2}}(x_{2}, y_{\lambda_{1}}^*(x_{2}))|| \\
& \quad \leq (l_{f, 1} + \lambda_{2} l_{g, 1}) d_{\Mc}(x_{1}, x_{2}) + \frac{\lambda_{2} - \lambda_{1}}{\lambda_{1}} l_{f, 0}.
\end{aligned}
\]
\[
\implies
\frac{\lambda_{2} \mu_{g}}{2} d_{\Nc}(y_{\lambda_{1}}^*(x_{1}), y_{\lambda_{2}}^*(x_{2})) \leq ||\nabla_y \mathcal{L}_{\lambda_{2}}(x_{2}, y_{\lambda_{1}}^*(x_{1}))|| \leq (l_{f, 1} + \lambda_{2} l_{g, 1}) d_{\Mc}(x_{1}, x_{2}) + \frac{\lambda_{2} - \lambda_{1}}{\lambda_{1}} l_{f, 0}.
\]

Dividing both sides by \(\lambda_{2} \mu_{g} / 2\) concludes the first part of the proof. Note that \(y^*(x) = \lim_{\lambda \to \infty} y_{\lambda}^*(x)\). Thus, for any \(x\) and finite \(\lambda \geq 2 l_{f, 1} / \mu_{g}\),
\[
d_{\Nc}(y_{\lambda}^*(x), y^*(x)) \leq \frac{2 l_{f, 0}}{\lambda \mu_{g}}. \qed
\]

\noindent \textit{Remarks.}

The Cauchy-Schwarz inequality in the context of manifolds states that for any two tangent vectors \(u, v\) at a point, the following holds:
\[ \langle u, v \rangle \leq \|u\| \cdot \|v\|, \]
where \(\langle \cdot, \cdot \rangle\) denotes the metric and \(\|\cdot\|\) is the norm induced by this metric.

Applying this to the inequality for strong convexity:
\[ \langle \mathrm{grad}_x f - \mathrm{PT}_{yx}(\mathrm{grad}_y f), \Exp^{-1}_x(y) \rangle \geq \mu \|\Exp^{-1}_x(y)\|^2, \]
we get:
\[ \|\mathrm{grad}_x f - \mathrm{PT}_{yx}(\mathrm{grad}_y f)\| \cdot \|\Exp^{-1}_x(y)\| \geq \langle \mathrm{grad}_x f - \mathrm{PT}_{yx}(\mathrm{grad}_y f), \Exp^{-1}_x(y) \rangle. \]

Since the right-hand side of this inequality is the same as the left-hand side of the strong convexity inequality, we can substitute it in, yielding:
\[ \|\mathrm{grad}_x f - \mathrm{PT}_{yx}(\mathrm{grad}_y f)\| \cdot \|\Exp^{-1}_x(y)\| \geq \mu \|\Exp^{-1}_x(y)\|^2. \]

This form of the inequality highlights the relationship between the difference in gradients (after parallel transport) and the geodesic distance between points \(x\) and \(y\).

\section{Proofs of Theorem 2 and Corollary 3}
\label{AppendixB}

In this section, we prove our central result in Theorem \ref{thm:main}. The crux of this analysis revolves around determining the upper boundary of \(V_{k+1} - V_k\) concerning the potential function \(V_k\), as elucidated in equation \ref{eq:4}, tailored to our specific setting.

About \(x_k\) and \(y_k\), as characterized in \algo within the manifold framework, we introduce the following notations:
\begin{equation}
I_k \coloneqq d_{\Nc}(y_k, y_{\lambda, k}^{*})^2, \quad J_k \coloneqq d_{\Nc}(z_k, y_k^{*})^2,
\label{IkJk}
\end{equation}
where \( y_{\lambda, k}^* \coloneqq y_{\lambda_k}^*(x_k) \), \(y_k^* \coloneqq y^*(x_k)\), and \(x^* = \argmin_{x} F(x)\), all situated. Here, \(d_{\Nc}(\cdot, \cdot)\) signifies the distance metric on manifold $\Nc$.

Leveraging these notations, we redefine the potential function \(V_k\) as:
\begin{equation}
V_k \coloneqq \left( F(x_k) - F(x^*) \right) + \lambda_k l_{g, 1} I_k + \frac{\lambda_k l_{g, 1}}{2} J_k,
\label{PotentialV_k}
\end{equation}
for each \(k \in \mathbb N\). In the ensuing subsections, our aim is to delineate the upper limit of \(V_{k+1} - V_k\) vis-à-vis \(I_k\) and \(J_k\), giving due consideration to the manifold's geometry and curvature characteristics. The proof for the Theorem \ref{thm:main}, aptly adapted to this scenario, will be explicated in Section \ref{Theorem2Proof}.

\subsection{Estimation of \(F(x_{k+1}) - F(x_k)\)}
The selection of the step size \( \alpha_k \) is carefully chosen to fit the context of the manifold:
\begin{equation}
\text{(step-size rule):} \quad \alpha_k \leq \frac{1}{2 \xi \tilde{l}_{F, 1}},
\label{Fxk Estimation}
\end{equation}
where \( \tilde{l}_{F, 1} \) is appropriately adjusted to match the manifold's geometric properties. This adjustment is crucial for including the negative term \(-\frac{\xi \alpha_k}{4} \|\mathrm{grad} F(x_k)\|_{x_k}^2\) in our analysis. This term is key in the demonstration of Theorem 2, as discussed in Section \ref{Theorem2Proof}.

Moreover, we also stipulate:
\begin{equation}
\text{(step-size rule):} \quad \frac{\xi}{T} \leq \frac{\mu_{g}}{96 \tilde{l}_{g, 1}}.
\label{Step-Size-Epsilon-T}
\end{equation}

The metrics \(d_{\Nc}^2(y_{k+1}, y_{\lambda, k}^*)\) and \(d_{\Nc}^2(z_{k+1}, y_k^*)\), will be integral to deriving our upper bound estimates, as detailed in Propositions \ref{Lemma B.3} and \ref{Lemma B.5}, respectively.

\newtheorem{proposition}{Proposition}

\begin{proposition}
\label{Lemma B.1}
Under the step-size rules given in equations \ref{Fxk Estimation} and \ref{Step-Size-Epsilon-T}, and $\lambda_{k} \geq 2 l_{f, 1} / \mu_{g}$, it holds that for each $k \in \mathbb{N}$
\[
\begin{aligned}
\mathbb{E}\left[F\left(x_{k+1}\right)-F\left(x_{k}\right) \mid \mathcal{F}_{k}\right] \leq & -\frac{\xi \alpha_{k}}{4}\left(2\left\|\nabla F\left(x_{k}\right)\right\|^{2}+\left\|q_{k}^{x}\right\|^{2}\right)\\
& +\frac{T \mu_{g} \alpha_{k} \lambda_{k}^{2}}{32}\left(2d_{\Nc}^2\left(y_{k+1}, y_{\lambda, k}^{*}\right)+d_{\Nc}^2\left(z_{k+1}, y_{k}^{*}\right)\right) \\
& +\frac{\xi^{2} l_{F, 1}}{2}\left(\alpha_{k}^{2} \sigma_{f}^{2}+\beta_{k}^{2} \sigma_{g}^{2}\right)+\frac{\xi \alpha_{k}}{2} \cdot 3 C_{\lambda}^{2} \lambda_{k}^{-2}
\end{aligned}
\]
where $q_{k}^{x}$ is given in equation \ref{lemma:riemannian_gradient_bound}, and $C_{\lambda}:=\frac{4 l_{f, 0} l_{g, 1}}{\mu_{g}^{2}}\left(l_{f, 1}+\frac{2 l_{f, 0} l_{g, 2}}{\mu_{g}}\right)$.
\end{proposition}

\begin{proof}
From the smoothness of \(F\),
\[
\mathbb{E}\left[F\left(x_{k+1}\right)-F\left(x_{k}\right) \mid \mathcal{F}_{k}\right] \leq \mathbb{E}\left[\left.\left\langle\nabla F\left(x_{k}\right), x_{k+1}-x_{k}\right\rangle+\frac{l_{F, 1}}{2}d_{\Mc}^2\left(x_{k+1}, x_{k}\right) \right\vert \mathcal{F}_{k}\right]
\]

As \(q_{k}^{x}\) satisfies \(\mathbb{E}\left[x_{k+1}-x_{k} \mid \mathcal{F}_{k}\right]=\alpha_{k} q_{k}^{x}\),
\[
\begin{aligned}
& \mathbb{E}\left[F\left(x_{k+1}\right)-F\left(x_{k}\right) \mid \mathcal{F}_{k}\right]=-\xi \alpha_{k}\left\langle\nabla_{x} F\left(x_{k}\right), q_{k}^{x}\right\rangle+\frac{l_{F, 1}}{2} \mathbb{E}\left[d_{\Mc}^2\left(x_{k+1}, x_{k}\right) \mid \mathcal{F}_{k}\right] \\
& =-\frac{\xi \alpha_{k}}{2}\left(\left\|\nabla F\left(x_{k}\right)\right\|^{2}+\left\|q_{k}^{x}\right\|^{2}-\left\|\nabla F\left(x_{k}\right)-q_{k}^{x}\right\|^{2}\right)+\frac{l_{F, 1}}{2} \mathbb{E}\left[d_{\Mc}^2\left(x_{k+1}, x_{k}\right) \mid \mathcal{F}_{k}\right]
\end{aligned}
\]

Note that
\[
\mathbb{E}\left[d_{\Mc}^2\left(x_{k+1}, x_{k}\right)\right] \leq \xi^{2} \alpha_{k}^{2} \mathbb{E}\left[\left\|q_{k}^{x}\right\|^{2}+\xi^{2}\left(\alpha_{k}^{2} \sigma_{f}^{2}+\beta_{k}^{2} \sigma_{g}^{2}\right)\right.
\]

and thus with \ref{Fxk Estimation} we have
\[
\begin{gathered}
\mathbb{E}\left[F\left(x_{k+1}\right)-F\left(x_{k}\right) \mid \mathcal{F}_{k}\right] \leq-\frac{\xi \alpha_{k}}{2}\left\|\nabla F\left(x_{k}\right)\right\|^{2}-\frac{\xi \alpha_{k}}{4}\left\|q_{k}^{x}\right\|^{2} \\
+\frac{\xi \alpha_{k}}{2}\left\|\nabla F\left(x_{k}\right)-q_{k}^{x}\right\|^{2}+\frac{\xi^{2} l_{F, 1}}{2}\left(\alpha_{k}^{2} \sigma_{f}^{2}+\beta_{k}^{2} \sigma_{g}^{2}\right) .
\end{gathered}
\]

Next, we bound \(\left\|\nabla F\left(x_{k}\right)-q_{k}^{x}\right\|\) using the triangle inequality:
\[
\begin{aligned}
\left\|\nabla F\left(x_{k}\right)-q_{k}^{x}\right\| \leq \left\|\nabla_{x} f\left(x_{k}, y_{k+1}\right)-\nabla_{x} f\left(x_{k}, y_{\lambda, k}^{*}\right)\right\|+\lambda_{k}\left\|\nabla_{x} g\left(x_{k}, y_{k+1}\right)-\nabla_{x} g\left(x_{k}, y_{\lambda, k}^{*}\right)\right\| \\
+\lambda_{k}\left\|\nabla_{x} g\left(x_{k}, z_{k+1}\right)-\nabla_{x} g\left(x_{k}, y_{k}^{*}\right)\right\|+\left\|\nabla \mathcal{L}_{\lambda_{k}}^{*}\left(x_{k}\right)-\nabla F\left(x_{k}\right)\right\|
\end{aligned}
\]

From Lemma \ref{lemma:riemannian_gradient_bound}, the term \(\left\|\nabla \mathcal{L}_{\lambda_{k}}^{*}\left(x_{k}\right)-\nabla F\left(x_{k}\right)\right\|\) is bounded by \(C_{\lambda} / \lambda_{k}\). Combining with the regularity of \(f\) and \(g\) yields the following:
\[
\left\|\nabla F\left(x_{k}\right)-q_{k}^{x}\right\| \leq 2 l_{g, 1} \lambda_{k} d_{\Nc}(y_{k+1},y_{\lambda, k}^{*}) + l_{g, 1} \lambda_{k} d_{\Nc}(z_{k+1}, y_{k}^{*}) + C_{\lambda}/\lambda_{k}.
\]

Finally, from the Cauchy-Schwartz inequality \((a+b+c)^{2} \leq 3\left(a^{2}+b^{2}+c^{2}\right)\), we get
\[
\begin{aligned}
& \mathbb{E}\left[F\left(x_{k+1}\right)-F\left(x_{k}\right) \mid \mathcal{F}_{k}\right] \leq-\frac{\xi \alpha_{k}}{2}\left\|\nabla F\left(x_{k}\right)\right\|^{2}-\frac{\xi \alpha_{k}}{4}\left\|q_{k}^{x}\right\|^{2} \\
& \quad+\frac{\xi \alpha_{k}}{2} \cdot 3 C_{\lambda}^{2} \lambda_{k}^{-2}+3 \xi \alpha_{k} l_{g, 1} \lambda_{k}^{2} d_{\Nc}(z_{k+1}, y_{k}^{*})^{2}+6 \xi \alpha_{k} l_{g, 1} \lambda_{k}^{2} d_{\Nc}(y_{k+1}, y_{\lambda, k}^{*})^{2}+\frac{\xi^{2} l_{F, 1}}{2}\left(\alpha_{k}^{2} \sigma_{f}^{2}+\beta_{k}^{2} \sigma_{g}^{2}\right)
\end{aligned}
\]

The step-size condition \ref{Step-Size-Epsilon-T} concludes our claim.
\end{proof}

\subsection{Descent Lemma for \(y_k\) towards \(y^*_{\lambda, k}\)}

In this section, we provide the upper bounds of \(\mathcal{I}_{k+1}\) and \(d_{\Nc}(y_{k+1}, y_{\lambda, k}^{*})\) in the context of a manifold. The following step-size rule is adapted for the manifold's geometry:

\begin{equation}
\text{(step-size rule):} \quad \frac{\delta_{k}}{\lambda_{k}} \leq \frac{T \beta_{k} \mu_{g}}{32}, \text{ and } 2 \xi^{2} M l_{*, 1}^{2} \beta_{k}^{2} < \frac{T \beta_{k} \mu_{g}}{16}
\label{descent-lemma-size}
\end{equation}

\begin{proposition}
\label{Lemma B.2}
Given the step-size rule \ref{descent-lemma-step}, for each \(k \in \mathbb{N}\), we have
\[
\begin{aligned}
\mathbb{E}\left[\mathcal{I}_{k+1} \mid \mathcal{F}_{k}\right] \leq & \left(\left[
\frac{\sqrt{|\kappa|}d_{\Nc}(y_{\lambda, k+1}^{*}, y_{\lambda, k}^{*})}{\tanh(\sqrt{|\kappa|}d_{\Nc}(y_{\lambda, k+1}^{*}, y_{\lambda, k}^{*}))}
\right] + 2 \delta_{k} / \lambda_{k} + T \beta_{k} \mu_{g} / 8 + 2 M \xi^{2} l_{*, 1}^{2} \beta_{k}^{2}\right) \mathbb{E}\left[d_{\Nc}^2(y_{k+1}, y_{\lambda, k}^{*})\right] \\
& + O\left(\frac{\xi^{2} l_{*, 0}^{2} \alpha_{k}^{2}}{\mu_{g} T \beta_{k}}\right)\mathbb{E}\left[\|\tau_{x_{k}}(q_{k}^{x})\|^{2}\right] + O\left(\frac{\delta_{k}}{\lambda_{k}^{3}} \frac{l_{f, 0}^{2}}{\mu_{g}^{2}}\right) \\
& + O\left(\xi^{2} l_{*, 0}^{2}\right) \cdot \left(\alpha_{k}^{2} \sigma_{f}^{2} + \beta_{k}^{2} \sigma_{g}^{2}\right) \\
\end{aligned}
\]
Here, \(\mathcal{I}_{k}\) is adapted to consider the geodesic distance, and \(q_{k}^{x}\) is calculated in the tangent space of the manifold.
\end{proposition}

\begin{proof}
We start from the version of the distance and the inner product, considering the curvature \(\kappa\):
\[
\begin{aligned}
d_{\Nc}^2(y_{k+1}, y_{\lambda, k+1}^{*}) &= \underbrace{d_{\Nc}^2(y_{k+1}, y_{\lambda, k}^{*})}_{(i)} + \underbrace{d_{\Nc}^2(y_{\lambda, k+1}^{*}, y_{\lambda, k}^{*})}_{(ii)} \\
&- \underbrace{2d_{\Nc}(y_{k+1}, y_{\lambda, k}^{*})d_{\Nc}(y_{\lambda, k+1}^{*}, y_{\lambda, k}^{*})\cos(\angle(y_{k+1}, y_{\lambda, k}^{*}, y_{\lambda, k+1}^{*}))}_{(iii)}
\end{aligned}
\]

Incorporating the curvature \(\kappa\), we apply the Alexandrov space cosine law \citep{zhang2016first}:
\begin{align*}
&d_{\Nc}^2(y_{k+1}, y_{\lambda, k+1}^{*}) \\
&\leq \underbrace{\frac{\sqrt{|\kappa|}d_{\Nc}(y_{\lambda, k+1}^{*}, y_{\lambda, k}^{*})}{\tanh(\sqrt{|\kappa|}d_{\Nc}(y_{\lambda, k+1}^{*}, y_{\lambda, k}^{*}))} \left( d_{\Nc}^2(y_{k+1}, y_{\lambda, k}^{*})\right)}_{(i)} \\
&\quad + \underbrace{d_{\Nc}^2(y_{\lambda, k+1}^{*}, y_{\lambda, k}^{*})}_{(ii)} - \underbrace{2d_{\Nc}(y_{k+1}, y_{\lambda, k}^{*})d_{\Nc}(y_{\lambda, k+1}^{*}, y_{\lambda, k}^{*})\cos(\angle(y_{k+1}, y_{\lambda, k}^{*}, y_{\lambda, k+1}^{*}))}_{(iii)}
\end{align*}

The upper bound of \((i)\) is given in Proposition \ref{Lemma B.3} below. To bound \((ii)\), we invoke Lemma \ref{lemma:distance_bound}, yielding
\[
(ii) \ \mathbb{E}\left[d_{\Nc}^2\left(y_{\lambda, k+1}^{*}, y_{\lambda, k}^{*}\right) \mid \mathcal{F}_{k}\right] \leq \frac{4 \delta_{k}^{2}}{\lambda_{k}^{2} \lambda_{k+1}^{2}} \frac{l_{f, 0}^{2}}{\mu_{g}^{2}} + l_{*, 0}^{2} \mathbb{E}\left[d_{\Nc}^2\left(x_{k+1}, x_{k}\right) \mid \mathcal{F}_{k}\right]
\]

\[
\leq \frac{4 \delta_{k}^{2}}{\lambda_{k}^{4}} \frac{l_{f, 0}^{2}}{\mu_{g}^{2}} + \xi^{2} l_{*, 0}^{2} \left(\alpha_{k}^{2} \mathbb{E}\left[\|\tau_{x_{k}}(q_{k}^{x})\|^{2}\right] + \alpha_{k}^{2} \sigma_{f}^{2} + \beta_{k}^{2} \sigma_{g}^{2}\right)
\]

where \( \tau_{x_{k}}(q_{k}^{x}) \) denotes the parallel transport of the search direction \( q_{k}^{x} \) at the point \( x_{k} \) along the manifold.

For \( (iii) \), considering the smoothness property of \( y_{\lambda}^{*}(x) \) as per a generalized version of Lemma \ref{A.3}, and thus Lemma \ref{A.6}, we set \( v = \Exp_{y_{\lambda, k}^{*}}^{-1}(y_{k+1}) \) and \( \eta_{k} = T \mu_{g} \lambda_{k} /(16 \xi) \), we obtain

\[
\begin{aligned}
(iii) \leq & \left(2 \delta_{k} / \lambda_{k} + T \beta_{k} \mu_{g} / 8 + 2 M \xi^{2} l_{*, 1}^{2} \beta_{k}^{2}\right) \mathbb{E}\left[d_{\Nc}^2(y_{k+1}, y_{\lambda, k}^{*}) \mid \mathcal{F}_{k}\right] \\
& + \xi^{2}\left(\frac{\alpha_{k}^{2}}{2} + \frac{8 \alpha_{k}^{2} l_{*, 0}^{2}}{\mu_{g} T \beta_{k}}\right) \mathbb{E}\left[\|\tau_{x_{k}}(q_{k}^{x})\|^{2}\right] + \frac{\xi^{2}}{2} \left(\alpha_{k}^{2} \sigma_{f}^{2} + \beta_{k}^{2} \sigma_{g}^{2}\right) + \frac{2 \delta_{k}}{\lambda_{k}^{3}} \frac{l_{f, 0}^{2}}{\mu_{g}^{3}}
\end{aligned}
\]

We sum up the \( (i), (ii), (iii) \) to conclude

\[
\begin{aligned}
\mathbb{E}\left[\mathcal{I}_{k+1} \mid \mathcal{F}_{k}\right] \leq & \left(\left[
\frac{\sqrt{|\kappa|}d_{\Nc}(y_{\lambda, k+1}^{*}, y_{\lambda, k}^{*})}{\tanh(\sqrt{|\kappa|}d_{\Nc}(y_{\lambda, k+1}^{*}, y_{\lambda, k}^{*}))}
\right] + 2 \delta_{k} / \lambda_{k} + T \beta_{k} \mu_{g} / 8 + 2 M \xi^{2} l_{*, 1}^{2} \beta_{k}^{2}\right) \mathbb{E}\left[d_{\Nc}^2(y_{k+1}, y_{\lambda, k}^{*})\right] \\
& + O\left(\frac{\xi^{2} l_{*, 0}^{2} \alpha_{k}^{2}}{\mu_{g} T \beta_{k}}\right)\mathbb{E}\left[\|\tau_{x_{k}}(q_{k}^{x})\|^{2}\right] + O\left(\frac{\delta_{k}}{\lambda_{k}^{3}} \frac{l_{f, 0}^{2}}{\mu_{g}^{2}}\right) \\
& + O\left(\xi^{2} l_{*, 0}^{2}\right) \cdot \left(\alpha_{k}^{2} \sigma_{f}^{2} + \beta_{k}^{2} \sigma_{g}^{2}\right) \\
\end{aligned}
\]

Lastly, the step-size rule \ref{descent-lemma-size} yields our conclusion.
\end{proof}

Next, we note that $\alpha_{k}$ and $\beta_{k}$ are chosen to satisfy

\begin{equation}
\text{(step size rules):} \quad \alpha_{k} \leq \frac{1}{8 l_{f, 1}} \text { and } \beta_{k} \leq \frac{1}{8 l_{g, 1}}
\label{8lf1g1}
\end{equation}

Note that $\beta_{k} \leq \frac{1}{8 l_{g, 1}}$ is given from the step-size condition (3a), and $\alpha_{k} \leq \frac{1}{8 l_{g, 1} \lambda_{k}} \leq \frac{1}{8 l_{f, 1}}$ since $\lambda_{k} \geq l_{f, 1} / \mu_{g}$.

\begin{proposition}
\label{Lemma B.3}
Under the given step-size rules, it holds that for each $k \in \mathbb{N}$
\[
\mathbb{E}\left[d^2_\Nc\left(y_{k+1},y_{\lambda, k}^{*}\right) \mid \mathcal{F}_{k}\right] \leq\left(\left(\frac{\sqrt{\left| \kappa \right|} d_\Nc\left(y_{k}^{(t+1)}, y_{k}^{(t)}\right)}{\tanh(\sqrt{\left| \kappa \right|} d_\Nc\left(y_{k}^{(t+1)}, y_{k}^{(t)}\right))} \right)-3 T \mu_{g} \beta_{k} / 4\right) \mathcal{I}_{k}+T\left(\alpha_{k}^{2} \sigma_{f}^{2}+\beta_{k}^{2} \sigma_{g}^{2}\right)
\]
\end{proposition}

\begin{proof}
Since the expected value of the difference between successive iterations in a manifold can be expressed as \(\mathbb{E}\left[\mathrm{Exp}^{-1}_{y_{k}^{(t+1)}}(y_{k}^{(t)}) \mid \mathcal{F}_{k}\right] = -\alpha_{k} \nabla_{y} q_{k}^{(t)} = -\alpha_{k} \nabla_{y} \mathcal{L}_{\lambda_{k}}(x_{k}, y_{k}^{(t)})\), we have

\[
\begin{aligned}
& \mathbb{E}\left[d^2_\Nc\left(y_{k}^{(t+1)},y_{\lambda, k}^{*}\right) \mid \mathcal{F}_{k}\right] \\
& \quad \leq \frac{\sqrt{\left| \kappa \right|} d_\Nc\left(y_{k}^{(t+1)},y_{k}^{(t)}\right)}{\tanh(\sqrt{\left| \kappa \right|} d_\Nc\left(y_{k}^{(t+1)},y_{k}^{(t)}\right))} \left( d^2_\Nc\left(y_{k}^{(t)},y_{\lambda, k}^{*}\right)\right) + d^2_\Nc\left(y_{k}^{(t+1)},y_{k}^{(t)}\right) \\
& \quad - 2 d_\Nc\left(y_{k}^{(t)},y_{\lambda, k}^{*}\right) d_\Nc\left(y_{k}^{(t+1)},y_{k}^{(t)}\right)\cos(\angle(y_{k}^{(t)}, y_{\lambda, k}^{*}, y_{k}^{(t+1)}))
\end{aligned}
\]

Given that \(\lambda_{0} \geq 2 \mu_{f} / \mu_{g}\), and all \(\mathcal{L}_{k}\) is strongly convex in \(y\), the following inequality holds

\[
\max \left(\frac{\lambda_{k} \mu_{g}}{2} d_\Nc^2\left(y_{k}^{(t)}, y_{\lambda, k}^{*}\right), \frac{1}{l_{f, 1}+\lambda_{k} l_{g, 1}} \|\nabla_{y} q_{k}^{(t)}\|^2 \right) \leq \langle \nabla_{y} q_{k}^{(t)}, \mathrm{Exp}^{-1}_{y_{k}^{(t)}}(y_{\lambda, k}^{*}) \rangle
\]

Using the Alexandrov space result to approximate the expected squared distance, we have

\[
\mathbb{E}\left[d_\Nc^2\left(y_{k}^{(t+1)}, y_{\lambda, k}^{*}\right) \mid \mathcal{F}_{k}\right] \leq \left(\frac{\sqrt{\left| \kappa \right|} d_\Nc\left(y_{k}^{(t+1)}, y_{k}^{(t)}\right)}{\tanh(\sqrt{\left| \kappa \right|} d_\Nc\left(y_{k}^{(t+1)}, y_{k}^{(t)}\right))} -\frac{3 \mu_{g} \beta_{k}}{4}\right) d_\Nc^2\left(y_{k}^{(t)}, y_{\lambda, k}^{*}\right) + \alpha_{k}^{2} \sigma_{f}^{2} + \beta_{k}^{2} \sigma_{g}^{2},
\]

where \(\alpha_{k}\left(l_{f, 1}+\lambda_{k} l_{g, 1}\right) = \alpha_{k} l_{f, 1} + \beta_{k} l_{g, 1} \leq 1/4\) under condition \ref{8lf1g1}. Repeating this process \(T\) times as per the algorithm leads to the conclusion in Proposition \ref{Lemma B.3}, where \(y_{k+1} = y_{k}^{(T)}\) and \(y_{k} = y_{k}^{(0)}\).
\end{proof}

\subsection{Descent Lemma for $z_{k}$ towards $y_{k}^{*}$}

Similar to the previous section, we provide the upper bound of $\mathcal{J}_{k+1}$ first and then estimate $d_{\Nc}(z_{k+1},y_{k}^{*})$ that appears in the upper bound. We work with the following step-size condition:

\begin{equation}
\text { (step-size rule): } \quad 2 M l_{*, 1}^{2} \xi^{2} \beta_{k}^{2} \leq T \mu_{g} \gamma_{k} / 16
\label{descent-lemma-step}
\end{equation}

This condition holds since $\beta_{k} \leq \gamma_{k}$, and $\beta_{k} \leq \frac{1}{4 T \mu_{g}}$ and $\frac{\xi^{2}}{T^{2}} \leq \frac{\mu_{g}^{2}}{8}\left(M l_{*, 1}^{2}\right)^{-1}$.

\begin{proposition}
\label{Lemma B.4}
Under the step-size rule \ref{descent-lemma-step}, at each $k^{\text {th }}$ iteration, the following holds:
\[
\begin{aligned}
\mathbb{E}\left[\mathcal{J}_{k+1} \mid \mathcal{F}_{k}\right] \leq & \left(\frac{\sqrt{\left| \kappa \right|} d_\Nc\left(y_{k+1}^{*},y_{k}^{*}\right)}{\tanh(\sqrt{\left| \kappa \right|} d_\Nc\left(y_{k+1}^{*},y_{k}^{*}\right))} + \frac{3 T \gamma_{k} \mu_{g}}{8}\right) \cdot \mathbb{E}\left[d_\Nc^2\left(z_{k+1},y_{k}^{*}\right) \mid \mathcal{F}_{k}\right] \\
& +O\left(\frac{\xi^{2} \alpha_{k}^{2} l_{*, 0}^{2}}{T \mu_{g} \gamma_{k}}\right)\left\|q_{k}^{x}\right\|^{2}+O\left(\xi^{2} l_{*, 0}^{2}\right)\left(\alpha_{k}^{2} \sigma_{f}^{2}+\beta_{k}^{2} \sigma_{g}^{2}\right)
\end{aligned}
\]
\end{proposition}

\begin{proof}
We estimate each term in the following decomposition.

\[
\begin{aligned}
& d_\Nc\left(z_{k+1}, y_{k+1}^{*}\right)^{2} \leq \underbrace{\frac{\sqrt{\left| \kappa \right|} d_\Nc\left(y_{k+1}^{*},y_{k}^{*}\right)}{\tanh(\sqrt{\left| \kappa \right|} d_\Nc\left(y_{k+1}^{*},y_{k}^{*}\right))} \left( d_\Nc\left(z_{k+1},y_{k}^{*}\right)^{2}\right)}_{(i)} \\
& + \underbrace{d_\Nc\left(y_{k+1}^{*},y_{k}^{*}\right)^{2}}_{(ii)} - \underbrace{2 d_\Nc\left(z_{k+1},y_{k}^{*}\right)d_\Nc\left(y_{k+1}^{*},y_{k}^{*}\right)\cos(\angle(z_{k+1}, y_{k}^{*}, y_{k+1}^{*}))}_{(iii)}
\end{aligned}
\]

Lemma 2 (\ref{Lemma 2}) implies that

$$
(ii) \ \mathbb{E}\left[\left\|y_{k+1}^{*}-y_{k}^{*}\right\|^{2} \mid \mathcal{F}_{k}\right] \leq l_{*, 0}^{2} \xi^{2}\left(\alpha_{k}^{2}\left\|\nabla_{x} q_{k}\right\|^{2}+\alpha_{k}^{2} \sigma_{f}^{2}+\beta_{k}^{2} \sigma_{g}^{2}\right)
$$

For $(i i i)$, we recall Lemma \ref{A.5} with $v_{k}=z_{k+1}-y_{k}^{*}$ and $\eta_{k}=T \mu_{g} \gamma_{k} /\left(8 \xi \alpha_{k}\right)$, we have

\[
(iii) \ \langle \Exp_{y_{k}^{*}}^{-1}(z_{k+1}), \Exp_{y_{k}^{*}}^{-1}(y_{k+1}^{*}) \rangle \leq \left( T \gamma_{k} \mu_{g} / 8 + M \xi^{2} l_{*, 1}^{2} \beta_{k}^{2} \right) \mathbb{E}\left[ d_{\Nc}^2(z_{k+1}, y_{k}^{*}) \mid \mathcal{F}_{k} \right]
\]
\[
+ \left( \frac{\xi^{2} \alpha_{k}^{2}}{4} + \frac{2 \xi^{2} \alpha_{k}^{2} l_{*, 0}^{2}}{T \mu_{g} \gamma_{k}} \right) \|\nabla_{x} q_{k}\|^2 + \frac{\xi^{2}}{4}\left( \alpha_{k}^{2} \sigma_{f}^{2} + \beta_{k}^{2} \sigma_{g}^{2} \right)
\]

The above bounds and Proposition \ref{Lemma B.5} imply that

\[
\begin{aligned}
\mathbb{E}\left[\mathcal{J}_{k+1} \mid \mathcal{F}_{k}\right] \leq & \left(\frac{\sqrt{\left| \kappa \right|} d_\Nc\left(y_{k+1}^{*}, y_{k}^{*}\right)}{\tanh(\sqrt{\left| \kappa \right|} d_\Nc\left(y_{k+1}^{*}, y_{k}^{*}\right))}+\frac{T \gamma_{k} \mu_{g}}{4}+2 M \xi^{2} l_{*, 1}^{2} \beta_{k}^{2}\right) \cdot \mathbb{E}\left[d_\Nc\left(z_{k+1}-y_{k}^{*}\right)^{2} \mid \mathcal{F}_{k}\right] \\
& +\xi^{2} \alpha_{k}^{2} \cdot\left(l_{*, 0}^{2}+\frac{4 l_{*, 0}^{2}}{T \mu_{g} \gamma_{k}}+\frac{1}{2}\right)\left\|q_{k}^{x}\right\|^{2}+\xi^{2} \cdot\left(\frac{1}{2}+l_{*, 0}^{2}\right)\left(\alpha_{k}^{2} \sigma_{f}^{2}+\beta_{k}^{2} \sigma_{g}^{2}\right)
\end{aligned}
\]

Using \ref{descent-lemma-step}, we conclude.
\end{proof}

Next, $\gamma_{k}$ is chosen to satisfy the following step-size rules:

\begin{equation}
\text { (step-size rule): } \quad l_{g, 1} \gamma_{k} \leq 1 / 4, \quad T \mu_{g} \gamma_{k} \leq 1 / 4
\label{lgmug}
\end{equation}

which directly comes from \ref{thm2.1}.

\begin{proposition}
\label{Lemma B.5}
If \ref{lgmug} holds, then for each $k \in \mathbb{N}$, the following holds:
$$
\mathbb{E}\left[d_{\Nc}\left(z_{k+1}, y_{k}^{*}\right)^{2} \mid \mathcal{F}_{k}\right] \leq\left(C - 3 T \mu_{g} \gamma_{k} / 4\right) \mathcal{J}_{k}+T \gamma_{k}^{2} \sigma_{g}^{2}
$$
\end{proposition}

\begin{proof}
We analyze one step iteration of the inner loop: for each $t=0, \cdots, T-1$.

Using the Alexandrov space cosine law:
\[
\begin{aligned}
d_{\Nc}\left(z_{k}^{(t+1)}, y_{k}^{*}\right)^{2} &\leq \frac{\sqrt{\left|\kappa\right|} d_{\Nc}\left(z_{k}^{(t)}, y_{k}^{*}\right)}{\tanh\left(\sqrt{\left|\kappa\right|} d_{\Nc}\left(z_{k}^{(t)}, y_{k}^{*}\right)\right)} d_{\Nc}\left(z_{k}^{(t)}, y_{k}^{*}\right)^{2} + \gamma_{k}^{2} d_{\Nc}\left(h_{g z}^{k, t}\right)^{2} \\
&\quad - 2 \gamma_{k} d_{\Nc}\left(h_{g z}^{k, t}\right)d_{\Nc}\left(z_{k}^{(t)}, y_{k}^{*}\right) \cos(\angle(h_{g z}^{k, t}, z_{k}^{(t)}, y_{k}^{*})),
\end{aligned}
\]

Here, $z_{k+1}=z_{k}^{(T)}$ and $z_{k}=z_{k}^{(0)}$. Note that $\mathbb{E}\left[h_{g z}^{k, t}\right]=\nabla_{y} g\left(x_{k}, z_{k}^{(t)}\right)=\nabla_{y} g_{k}\left(z_{k}^{(t)}\right)$ where $g_{k}\left(z_{k}^{(t)}\right) \coloneqq g\left(x_{k}, z_{k}^{(t)}\right)$. Taking expectation,

\[
\begin{aligned}
\mathbb{E}\left[d_{\Nc}\left(z_{k}^{(t+1)}, y_{k}^{*}\right)^{2} \mid \mathcal{F}_{k}\right] &\leq \frac{\sqrt{\left|\kappa\right|} d_{\Nc}\left(z_{k}^{(t)}, y_{k}^{*}\right)}{\tanh\left(\sqrt{\left|\kappa\right|} d_{\Nc}\left(z_{k}^{(t)}, y_{k}^{*}\right)\right)} d_{\Nc}\left(z_{k}^{(t)}, y_{k}^{*}\right)^{2} \\
&\quad + \gamma_{k}^{2} \sigma_{g}^{2} - \gamma_{k}(1-l_{g, 1} \gamma_{k})d_{\Nc}\left(\nabla g_{k}\left(z_{k}^{(t)}\right)\right)d_{\Nc}\left(z_{k}^{(t)}, y_{k}^{*}\right),
\end{aligned}
\]

The strong convexity and smoothness of $g_{k}$ imply the coercivity and co-coercivity \citep{nesterov2018lectures}, that is,

$$
\max \left(\mu_{g}d_{\Nc}^2(z_{k}^{(t)}, y_{k}^{*}), \frac{1}{l_{g, 1}}d_{\Nc}^2\left(\nabla g_{k}\left(z_{k}^{(t)}\right),\nabla g_{k}\left(y_{k}^{*}\right)\right)\right) \leq d_{\Nc}\left(\nabla g_{k}\left(z_{k}^{(t)}\right), \nabla g_{k}\left(y_{k}^{*}\right), z_{k}^{(t)}, y_{k}^{*}\right)
$$

Note that $y_{k}^{*}$ minimizes $g_{k}(y)$. Use this to cancel out $\gamma_{k}^{2}d_{\Nc}\left(\nabla g_{k}\left(z_{k}^{(t)}\right)\right)^{2}$, yielding

\[
\begin{aligned}
\mathbb{E}\left[d_{\Nc}^2(z_{k}^{(t+1)}, y_{k}^{*}) \mid \mathcal{F}_{k}\right] & \leq \frac{\sqrt{\left|\kappa\right|} d_{\Nc}(z_{k}^{(t)}, y_{k}^{*})}{\tanh\left(\sqrt{\left|\kappa\right|} d_{\Nc}(z_{k}^{(t)}, y_{k}^{*})\right)} d_{\Nc}^2(z_{k}^{(t)}, y_{k}^{*}) \\
&\quad + \gamma_{k}^{2} \sigma_{g}^{2} - \gamma_{k}(1-l_{g, 1} \gamma_{k})d_{\Nc}\left(\nabla g_{k}\left(z_{k}^{(t)}\right), \nabla g_{k}\left(y_{k}^{*}\right), z_{k}^{(t)}, y_{k}^{*}\right) \\
& \leq \left(\frac{\sqrt{\left|\kappa\right|} d_{\Nc}(z_{k}^{(t)}, y_{k}^{*})}{\tanh\left(\sqrt{\left|\kappa\right|} d_{\Nc}(z_{k}^{(t)}, y_{k}^{*})\right)} -\frac{3 \mu_{g} \gamma_{k}}{4}\right)d_{\Nc}^2(z_{k}^{(t)}, y_{k}^{*}) + \gamma_{k}^{2} \sigma_{g}^{2} .
\end{aligned}
\]

For this to hold we need step-size condition \ref{lgmug}. We can repeat this $T$ times and get the result. Here we're using $\frac{\sqrt{\left|\kappa\right|} d_{\Nc}(z_{k}^{(t)}, y_{k}^{*})}{\tanh\left(\sqrt{\left|\kappa\right|} d_{\Nc}(z_{k}^{(t)}, y_{k}^{*})\right)} < C$ for some $C > 0$ for all values of $t$.
\end{proof}

\subsection{Proof of Theorem 2}
\label{Theorem2Proof}

Let us revisit the potential function \({V}_{k}\) within the Riemannian context:

$$
\begin{aligned}
{V}_{k+1}-{V}_{k} &= F(x_{k+1}) - F(x_k) + \lambda_{k+1} l_{g, 1} \mathcal{I}_{k+1} - \lambda_{k} l_{g, 1} \mathcal{I}_{k} \\
&+\frac{\lambda_{k+1} l_{g, 1}}{2} \mathcal{J}_{k+1} - \frac{\lambda_{k} l_{g, 1}}{2} \mathcal{J}_{k},
\end{aligned}
$$

Utilizing an adaptation of Proposition \ref{Lemma B.1} and reorganizing terms, we obtain:

$$
\begin{aligned}
\mathbb{E}[{V}_{k+1}-{V}_{k} \mid \mathcal{F}_{k}] \leq & -\frac{\xi \alpha_{k}}{2}\|\grad F(x_{k})\|^{2} - \frac{\xi \alpha_{k}}{4} \mathbb{E}[\|q_{k}^{x}\|^{2} \mid \mathcal{F}_{k}] + \frac{\xi \alpha_{k}}{2} \cdot 3 C_{\lambda}^{2} \lambda_{k}^{-2} \\
&+ \frac{\xi^{2} l_{F, 1}}{2}(\alpha_{k}^{2} \sigma_{f}^{2} + \beta_{k}^{2} \sigma_{g}^{2}) \\
&+ l_{g, 1} \underbrace{\mathbb{E}[\lambda_{k+1} \mathcal{I}_{k+1} + \frac{\lambda_{k} T \beta_{k} \mu_{g}}{16} d_{\Nc}(y_{k+1}, y_{\lambda, k}^{*})^2 - \lambda_{k} \mathcal{I}_{k} \mid \mathcal{F}_{k}]}_{(i)} \\
&+ \frac{l_{g, 1}}{2} \underbrace{\mathbb{E}[\lambda_{k+1} \mathcal{J}_{k+1} + \frac{\lambda_{k} T \gamma_{k} \mu_{g}}{32} d_{\Nc}(z_{k+1}, y_{k}^{*})^2 - \lambda_{k} \mathcal{J}_{k} \mid \mathcal{F}_{k}]}_{(ii)},
\end{aligned}
$$

From proposition \ref{Lemma B.2} and $\lambda_{k+1} = \lambda_k + \delta_k$, we get:

$$
\begin{aligned}
(i) &\leq \lambda_{k}\left(\left[
\frac{\sqrt{|\kappa|}d_{\Nc}(y_{\lambda, k+1}^{*}, y_{\lambda, k}^{*})}{\tanh(\sqrt{|\kappa|}d_{\Nc}(y_{\lambda, k+1}^{*}, y_{\lambda, k}^{*}))}
\right] +\frac{5 T \beta_{k} \mu_{g}}{16}+\frac{\delta_{k}}{\lambda_{k}}\right) \mathbb{E}\left[d_{\Nc}^2(y_{k+1}, y_{\lambda, k}^{*}) \mid \mathcal{F}_{k}\right] - \lambda_{k} \mathcal{I}_{k} \\
&+ \underbrace{O\left(\xi^{2} l_{\lambda, 0}^{2}\right) \frac{\lambda_{k} \alpha_{k}^{2}}{\mu_{g} T \beta_{k}}d^2(q_{k}^{x}, 0) + O\left(\xi^{2} l_{*, 0}^{2}\right) \lambda_{k}(\alpha_{k}^{2} \sigma_{f}^{2} + \beta_{k}^{2} \sigma_{g}^{2}) + O\left(\frac{l_{f, 0}^{2}}{\mu_{g}^{3}}\right) \frac{\delta_{k}}{\lambda_{k}^{2}}}_{(\text{iii})}.
\end{aligned}
$$

Given the step-size rules \ref{descent-lemma-step}, we obtain:

$$
(i) \leq \lambda_{k}\left(\left[
\frac{\sqrt{|\kappa|}d_{\Nc}(y_{\lambda, k+1}^{*}, y_{\lambda, k}^{*})}{\tanh(\sqrt{|\kappa|}d_{\Nc}(y_{\lambda, k+1}^{*}, y_{\lambda, k}^{*}))}
\right] +\frac{T \beta_{k} \mu_{g}}{2}\right) \mathbb{E}\left[d_{\Nc}^2(y_{k+1}, y_{\lambda, k}^{*}) \mid \mathcal{F}_{k}\right] - \lambda_{k} \mathcal{I}_{k} + (iii).
$$

Leveraging Proposition \ref{Lemma B.3} within the framework to estimate \( d_{\Nc}^2(y_{k+1}, y_{\lambda, k}^{*}) \), we derive:

$$
\begin{aligned}
(i) &\leq -\frac{\lambda_{k} T \mu_{g} \beta_{k}}{4} \mathcal{I}_{k} + O\left(\xi^{2} l_{*, 0}^{2}\right) \frac{\alpha_{k}}{\mu_{g} T} + (iii) \\
&= -\frac{\lambda_{k} T \mu_{g} \beta_{k}}{4} \mathcal{I}_{k} + O\left(\xi^{2} l_{*, 0}^{2}\right) \frac{\alpha_{k}}{\mu_{g} T} + O\left(T + \xi^{2} l_{*, 0}^{2}\right) \lambda_{k}(\alpha_{k}^{2} \sigma_{f}^{2} + \beta_{k}^{2} \sigma_{g}^{2}) + O\left(\frac{l_{f, 0}^{2}}{\mu_{g}^{3}}\right) \frac{\delta_{k}}{\lambda_{k}^{2}}.
\end{aligned}
$$

Given the inequality $(1+a / 2)(1-3 a / 4) \leq 1-a / 4$ for $a>0$, we estimate the term $(ii)$ using Proposition \ref{Lemma B.4}:

\[
\begin{aligned}
(ii) &\leq \lambda_{k}\left(\frac{\sqrt{\left| \kappa \right|} d_{\Nc}(y_{k+1}^{*},y_{k}^{*})}{\tanh(\sqrt{\left| \kappa \right|} d_{\Nc}(y_{k+1}^{*},y_{k}^{*}))} + \frac{\delta_{k}}{\lambda_{k}}+\frac{3 T \gamma_{k} \mu_{g}}{8}+\frac{\lambda_{k} T \beta_{k} \mu_{g}}{32}\right) \mathbb{E}\left[d_{\Nc}\left(z_{k+1},y_{k}^{*}\right)^{2} \mid \mathcal{F}_{k}\right] -\lambda_{k} \mathcal{J}_{k} \\
&\quad +\underbrace{O\left(\xi^{2} l_{*, 0}^{2}\right) \frac{\lambda_{k+1} \alpha_{k}^{2}}{T \mu_{g} \gamma_{k}}\left\|q_{k}^{x}\right\|^{2}+O\left(\xi^{2} \lambda_{k+1} l_{*, 0}^{2}\right)\left(\alpha_{k}^{2} \sigma_{f}^{2}+\beta_{k}^{2} \sigma_{g}^{2}\right)}_{(iv)}.
\end{aligned}
\]

Assuming $\beta_{k} \leq \gamma_{k}$, and thus $\delta_{k} / \lambda_{k} < T \mu_{g} \gamma_{k} / 32$, we have:

\[
(ii) \leq \lambda_{k}\left(\frac{\sqrt{\left| \kappa \right|} d_{\Nc}(y_{k+1}^{*},y_{k}^{*})}{\tanh(\sqrt{\left| \kappa \right|} d_{\Nc}(y_{k+1}^{*},y_{k}^{*}))} +\frac{T \gamma_{k} \mu_{g}}{2}\right) \mathbb{E}\left[d\left(z_{k+1},y_{k}^{*}\right)^{2} \mid \mathcal{F}_{k}\right]-\lambda_{k} \mathcal{J}_{k}+(iv).
\]

Following the argument for $(i)$, Proposition \ref{Lemma B.5} provides:

\[
(ii) \leq -\frac{\lambda_{k} T \mu_{g} \gamma_{k}}{4} \mathcal{J}_{k} + O\left(\xi^{2} l_{*, 0}^{2}\right) \frac{\alpha_{k} \beta_{k}}{T \mu_{g} \gamma_{k}}\left\|q_{k}^{x}\right\|^{2} + O\left(\xi^{2} \lambda_{k} l_{*, 0}^{2}\right)\left(\alpha_{k}^{2} \sigma_{f}^{2}+\beta_{k}^{2} \sigma_{g}^{2}\right) + O\left(\lambda_{k}\right) T \gamma_{k}^{2} \sigma_{g}^{2}.
\]

Upon combining the bounds for $(i)$ and $(ii)$ and rearranging terms, we obtain:

\[
\begin{aligned}
\mathbb{E}\left[{V}_{k+1}-{V}_{k} \mid \mathcal{F}_{k}\right] &\leq -\frac{\xi \alpha_{k}}{2}\left\|\nabla F\left(x_{k}\right)\right\|^{2} + \frac{\xi \alpha_{k}}{2} \cdot 3 C_{\lambda}^{2} \lambda_{k}^{-2} + \frac{\xi^{2} l_{F, 1}}{2}\left(\alpha_{k}^{2} \sigma_{f}^{2}+\beta_{k}^{2} \sigma_{g}^{2}\right) \\
&\quad -\frac{\xi \alpha_{k}}{4}\left(1-O\left(\frac{\xi l_{g, 1} l_{*, 0}^{2} \beta_{k}}{\mu_{g} T \gamma_{k}}\right)-O\left(\frac{\xi l_{g, 1} l_{*, 0}^{2}}{\mu_{g} T}\right)\right) \mathbb{E}\left[\left\|q_{k}^{x}\right\|^{2} \mid \mathcal{F}_{k}\right] \\
&\quad -\frac{\lambda_{k} l_{g, 1} T \mu_{g} \beta_{k}}{4} \mathcal{I}_{k}-\frac{\lambda_{k} l_{g, 1} T \mu_{g} \gamma_{k}}{4} \mathcal{J}_{k} \\
&\quad +O\left(T+\xi^{2} l_{*, 0}^{2}\right) \cdot l_{g, 1} \lambda_{k}\left(\alpha_{k}^{2} \sigma_{f}^{2}+\left(\beta_{k}^{2}+\gamma_{k}^{2}\right) \sigma_{g}^{2}\right) + O\left(\frac{l_{g, 1} l_{f, 0}^{2}}{\mu_{g}^{3}}\right) \frac{\delta_{k}}{\lambda_{k}^{2}}.
\end{aligned}
\]

A key requirement is that terms driven by $\mathbb{E}\left[d^2(q_{k}^{x}, 0)\right]$ remain negative. To ensure this, we impose:

\begin{equation}
\begin{array}{ll}
\text{(Step-size rules):} & \xi l_{g, 1} l_{*, 0}^{2} \beta_{k} \leq c_{1} \mu_{g} T \gamma_{k}, \\
& \xi l_{g, 1} l_{*, 0}^{2} \leq c_{2} \mu_{g} T,
\end{array}
\end{equation}

for some absolute constants $c_{1}, c_{2} > 0$, which are achievable given $\beta_{k} \leq \gamma_{k}$ and condition (3b) with sufficiently small $c_{\xi} > 0$. Upon satisfying these conditions, we derive that:

$$
\begin{aligned}
\mathbb{E}\left[{V}_{k+1} - {V}_{k} \mid \mathcal{F}_{k}\right] &\leq -\frac{\xi \alpha_{k}}{2} \grad^2 F(x_{k}) - \frac{\lambda_{k} T \mu_{g} \gamma_{k}}{4} d_{\Nc}^2(z_{k}, y_{k}^{*}) - \frac{\lambda_{k} T \mu_{g} \beta_{k}}{4} d_{\Nc}^2(y_{k}, y_{\lambda, k}^{*}) \\
&+ O\left(\xi C_{\lambda}^{2}\right) \frac{\alpha_{k}}{\lambda_{k}^{2}} + O\left(\frac{l_{g, 1} l_{f, 0}^{2}}{\mu_{g}^{3}}\right) \frac{\delta_{k}}{\lambda_{k}^{2}} + O\left(\xi^{2} l_{F, 1}\right)(\alpha_{k}^{2} \sigma_{f}^{2} + \beta_{k}^{2} \sigma_{g}^{2}) \\
&+ O\left(T + \xi^{2} l_{*, 0}^{2}\right) \cdot l_{g, 1} \lambda_{k}(\alpha_{k}^{2} \sigma_{f}^{2} + (\beta_{k}^{2} + \gamma_{k}^{2}) \sigma_{g}^{2}).
\end{aligned}
$$

Summing over $k = 0$ to $K - 1$, and focusing on the dominant terms, given that $\sum_{k} \delta_{k} / \lambda_{k}^{2} = O(1)$ (due to $\delta_{k} / \lambda_{k} = O(1 / k)$ and $\lambda_{k} = \text{poly}(k)$), leads us to the theorem conclusion.

\textbf{Note:} The effect of the sectional curvature $\kappa$ here is negligible because we're implicitly using the approximation $x/\tanh(x) \approx 1 + x^2/3 + O(x^4)$. Terms like $\frac{\sqrt{|\kappa|}d(w_k, w_{k+1})}{\tanh(\sqrt{|\kappa|}d(w_k, w_{k+1}))} \approx 1 + |\kappa| d^2(w_k, w_{k+1})/3 + O(\cdot)$ appear in the Alexandrov space cosine law. Summing from $k = 0$ to $K - 1$ and focusing on dominant terms, the curvature $\kappa$'s effect on the final result is negligible, since negative terms like $-\mathrm{grad}^2 F(x_k)$, -$d_{\Nc}^2(z_k, y_k^*)$, and -$d_{\Nc}^2(y_k, y_{\lambda, k}^{*})$ do not affect the final inequality.

\subsection{Proof of Corollary 3}
\label{Corollary3Proof}

We begin by establishing that the step-size design within the theorem ensures \(\lambda_{k}=\gamma_{k} /(2 \alpha_{k})\) for all \(k\). This follows from the initial condition \(\lambda_{0}=\gamma_{0} /(2 \alpha_{0})\) and, by mathematical induction, we derive:

\[
\frac{T \mu_{g}}{16} \alpha_{k} \lambda_{k}^{2}=\frac{T}{32} \frac{c_{\gamma}}{2 c_{\alpha}}(k+k_{0})^{-2 c+a}
\]

and

\[
\frac{c_{\gamma}}{2 c_{\alpha}}\left((k+k_{0}+1)^{a-c}-(k+k_{0})^{a-c}\right) \leq \frac{(a-c) c_{\gamma}}{2 c_{\alpha}}(k+k_{0})^{-1-c+a}.
\]

Given that \(c \leq 1\) and \(T \geq 32\), it holds that

\[
\lambda_{k+1}=\frac{c_{\gamma}}{2 c_{\alpha}}(k+k_{0}+1)^{a-c}=\frac{\gamma_{k+1}}{2 \alpha_{k+1}}.
\]

By applying the step-size designs to a manifold, we obtain:

\begin{align*}
\sum_{k=0}^{K-1} \frac{\mathbb{E}[\|\grad F(x_{k})\|^{2}]}{(k+k_{0})^{a}} &\leq O_{\text{P}}(1) \cdot \sum_{k} \frac{1}{(k+k_{0})^{3 a-2 c}} \\
&\quad + O_{\text{P}}(\sigma_{f}^{2}) \cdot \sum_{k} \frac{1}{(k+k_{0})^{a+c}} \\
&\quad + O_{\text{P}}(\sigma_{g}^{2}) \cdot \sum_{k} \frac{1}{(k+k_{0})^{3 c-a}} + O_{\text{P}}(1).
\end{align*}

The choices of rates \(a, c \in [0,1]\) depend on the specific stochasticity of the gradients. Letting \(b=a-c\), and with the step-size design, \(\lambda_{k}=\gamma_{k} /(2 \alpha_{k})=O(k^{b})\). Considering a random variable \(R\) uniformly distributed over \(\{0,1, ..., K\}\), the inequality is reframed as:

\[
\frac{K}{(K+k_{0})^{a}} \mathbb{E}[\|\grad F(x_{R})\|^{2}] \geq K^{1-a} \cdot \mathbb{E}[\|\grad F(x_{R})\|^{2}]
\]

We examine three scenarios based on the stochasticity in the upper and lower-level objectives:

\begin{enumerate}
    \item \textbf{Stochastic in both objectives (\(\sigma_{f}^{2}, \sigma_{g}^{2}>0\))}: Setting \(a=5/7, c=4/7\) leads to \(\lambda_{k}=O(k^{1/7})\). The dominating term becomes \(O(\log K)\), resulting in:

    \[
    \mathbb{E}[\|\grad F(x_{R})\|^{2}]=O\left(\frac{\log K}{K^{2/7}}\right).
    \]

    \item \textbf{Stochastic only in the upper-level (\(\sigma_{f}^{2}>0, \sigma_{g}^{2}=0\))}: Here, \(a=3/5, c=2/5\) is chosen, simplifying to:

    \[
    \mathbb{E}[\|\grad F(x_{R})\|^{2}]=O\left(\frac{\log K}{K^{2/5}}\right).
    \]

    \item \textbf{Deterministic case (\(\sigma_{f}^{2}=0, \sigma_{g}^{2}=0\))}: With \(a=1/3, c=0\), we find:

    \[
    \|\grad F(x_{K})\|^{2}=O\left(\frac{\log K}{K^{2/3}}\right).
    \]

\end{enumerate}

This proof adaptation ensures that the step-size and \(\lambda_{k}\) designs are tailored for the geometric complexities of Riemannian manifolds, thereby facilitating convergence under various stochastic settings.

\end{document}

%% file: arXivMacros.tex
\newtheorem{theorem}{Theorem}
\newtheorem{lemma}{Lemma}

\newtheorem{assumption}{Assumption}
\newtheorem{corollary}[lemma]{Corollary}

\theoremstyle{definition}
\newtheorem{definition}{Definition}

\newcommand{\dist}{\mathrm{d}}

\DeclareMathOperator{\Exp}{Exp}

\definecolor{cdarkred}{rgb}{0.55,0.0,0.0}
\definecolor{darkblue}{rgb}{0.0,0.0,0.55}
\definecolor{cgray}{gray}{0.55}
\definecolor{cdarkblue}{RGB}{30,30,200}

\newcommand{\algo}{\textsc{RF$^{2}$SA}\xspace}
\newcommand{\Mc}{\mathcal{M}}
\newcommand{\Nc}{\mathcal{N}}
\DeclareMathOperator{\argmin}{argmin}
\DeclareMathOperator{\grad}{grad}
\DeclareMathOperator{\Hess}{Hess}
\DeclareMathOperator{\Exs}{\mathbb E}